\documentclass[10pt,oneside]{article}

\usepackage[utf8]{inputenc} %
\usepackage[T1]{fontenc}    %
\usepackage{url}            %
\usepackage{booktabs}       %
\usepackage{amsfonts}       %
\usepackage{nicefrac}       %
\usepackage{microtype}      %

\usepackage{lmodern}
\usepackage{times}

\usepackage{amssymb,amsmath,amsthm}
\usepackage{bbold}
\usepackage[short]{optidef}

\usepackage{framed,multirow,multicol}

\usepackage{xcolor,xspace}
\usepackage{lscape}
\usepackage{subfigure,graphicx,epsfig,tikz,caption}
\usepackage[normalem]{ulem}
\usepackage{enumerate}
\usepackage{verbatim}
\usepackage{makeidx,latexsym}
\usepackage[colorlinks=true,citecolor=blue]{hyperref}%

\usepackage{bm}
\usepackage{stmaryrd}
\usepackage{lscape}

\usepackage[english]{babel}
\usepackage{bbm}
\usepackage{pgfplots}

\usepackage{parskip}
\usepackage[letterpaper,margin=1.1in]{geometry}

\RequirePackage[OT1]{fontenc}
\RequirePackage{amsthm,amsmath}

\usepackage{comment}
\usepackage{graphicx}
\usepackage{algorithm,algpseudocode}
\usepackage{bm}
\usepackage{amsfonts}
\usepackage{comment}
\numberwithin{equation}{section}
\theoremstyle{plain}
\usepackage{amssymb}
\newtheorem{lm}{Lemma}[section]
\newtheorem{cor}{Corollary}[section]
\newtheorem{df}{Definition}[section]
\newtheorem{pr}{Proposition}[section]

\usepackage{natbib}
\newcommand{\reals}{\mathbb{R}}
\renewcommand{\vec}{\mathbf }
\newcommand{\E}{\mathbb{E}}
\newcommand{\EP}{\mathbb{E}^0}
\renewcommand{\P}{\mathbb{P}}
\newcommand{\PP}{\mathbb{P}^0}
\newcommand{\ind}{\mathbb{I}}
\newcommand{\kl}{\mathrm{KL}}

\newtheorem{thm}{Theorem}[section]

\begin{document}

\title{Dynamics and Inference for Voter Model Processes}

\author{Milan Vojnovic and Kaifang Zhou\thanks{
Email correspondence: \{m.vojnovic,k.zhou3\}@lse.ac.uk}\\
Department of Statistics\\
London School of Economics and Political Science\\
London, WC2A 2AE, United Kingdom
}	

\date{}





\maketitle

\begin{abstract}
We consider a discrete-time voter model process on a set of nodes, each being in one of two states, either $0$ or $1$. In each time step, each node adopts the state of a randomly sampled neighbor according to sampling probabilities, referred to as node interaction parameters. We study the maximum likelihood estimation of the node interaction parameters from observed node states for a given number of realizations of the voter model process. In contrast to previous work on parameter estimation of network autoregressive processes, whose long-run behavior is according to a stationary stochastic process, the voter model is an absorbing stochastic process that eventually reaches a consensus state. This requires developing a framework for deriving parameter estimation error bounds from observations consisting of several realizations of a voter model process. We present parameter estimation error bounds by interpreting the observation data as being generated according to an extended voter process that consists of cycles, each corresponding to a realization of the voter model process until absorption to a consensus state. In order to obtain these results, consensus time of a voter model process plays an important role. We present new bounds for all moments and a bound that holds with any given probability for consensus time, which may be of independent interest. In contrast to most existing work, our results yield a consensus time bound that holds with high probability. We also present a sampling complexity lower bound for parameter estimation within a prescribed error tolerance for the class of locally stable estimators.

\end{abstract}


%

\section{Introduction} 
\label{sec:intro}
The mathematical models known as \emph{interacting particle systems} have been studied in different academic disciplines, with a canonical application to modeling opinion formation in social networks, where individuals interact pairwise and update their state in a way depending on their previous states \cite{aldous2013}. Models of opinion formation in social networks, e.g. \cite{D74}, were introduced to study how consensus is reached in a network where individuals update their opinions based on their personal preferences and observed opinions of their neighbors. Threshold models of collective behavior \cite{G78} assume individuals update their opinions according to a threshold rule, with an individual adopting a new state only if the number of its neighbors who adopted this state exceeds a threshold value. 

In this paper, we consider the classic interacting particle system known as the voter model. The voter model was introduced in \cite{holley1975} as a continuous-time Markov process, under which each individual is in one of two possible states, either $0$ or $1$. In this model, each individual adopts the state of a randomly sampled neighbor at random time instances according to independent Poisson processes associated with individuals or links connecting them. This voter model is an instance of an interacting particle system \cite{L85}. A \emph{noisy voter model} was introduced in \cite{G95}, which is obtained from the classic voter model by adding spontaneous flipping of states from $0$ to $1$, and $1$ to $0$, to each node. The discrete-time voter model is defined analogously to the continuous-time voter model, but with node states updated synchronously at discrete time steps. The discrete-time voter model was studied under different assumptions about which nodes update their states at discrete time points, e.g. \cite{NIY99}, \cite{HP02}, and \cite{CR16}. 


We study dynamics and inference for the discrete-time \emph{voter model}, defined as a Markov chain $\{X_t\}_{t\geq 0}$ with state space $\{0,1\}^n$, where $X_t$ represents states of nodes at time $t$, updated such that node states $X_{t+1}$ are independent conditional on $X_t$, with marginal distributions
\begin{equation}
X_{t+1,u} \mid X_t \sim \mathrm{Ber}(a_u^\top X_t) \hbox{ for } t\geq 0 \hbox{ and } u \in \{1,\ldots,n\}
\label{equ:v}
\end{equation}
where $a_u^\top$ is the $u$-th row of a stochastic matrix $A$, 
the initial state $X_0$ is assumed to have distribution $\mu$, and $\mathrm{Ber}(p)$ denotes Bernoulli distribution with mean $p$. A matrix is said to be a \emph{stochastic matrix} if it has real, non-negative elements and all row sums equal to $1$. Intuitively, $a_{u,v}$ is the probability of node $u$ sampling node $v$ in a time step. 

The voter model can be equivalently defined as a random linear dynamical system with $X_0\sim \mu$ and
\begin{equation}
X_{t+1}=Z_{t+1}X_t, \hbox{ for } t\geq 0
\label{equ:xzx}
\end{equation}
where $Z_1, Z_2, \ldots$ are independent and identically distributed (i.i.d.) $n\times n$ random stochastic matrices, with elements of value $0$ or $1$, and $\E[Z_1] = A$. 

The voter model has $C= \{\vec{0},\vec{1}\}$ as absorbing states and all other states are transient. Statistical inference for the voter model asks to estimate parameter $A$ from $m\geq 1$ independent sample paths of the voter model process. For the analysis of parameter estimation, it is convenient to consider an \emph{extended voter process} that consists of cycles, each of which corresponding to a realization of the voter model process with initial state sampled according to given initial state distribution and ending at hitting a consensus state. Such an extended voter process is defined as 
\begin{equation}
X_{t+1} = Z_{t+1}X_t \ind_{\{X_t\notin C\}} + \xi_{t+1}\ind_{\{X_t\in C\}}
\label{equ:ev}
\end{equation}
where $\xi_t$ is an i.i.d. sequence of random vectors taking values in $\{0,1\}^n$ according to distribution $\mu$.

The voter model defined by (\ref{equ:v}) and equivalently by (\ref{equ:xzx}) is defined such that all nodes update their states in every time step. We will also consider an \emph{asynchronous discrete-time voter model} under which in each time step exactly one node updates its state. The asynchronous voter model dynamics is defined by
$$
X_{t+1,u} \mid X_t \sim \mathrm{Ber}(a_u^\top X_t) \hbox{ if } u = I_t \hbox{ and } X_{t+1,u} = X_{t,u} \hbox{ if } u\in \{1,\ldots, n\}\setminus \{I_t\}
$$
where $I_t$ are i.i.d. random variables according to uniform distribution on $\{1,\ldots, n\}$. The asynchronous discrete-time voter model also obeys the random linear dynamical system recursive equation (\ref{equ:xzx}) but with $Z_1,Z_2,\ldots$ being i.i.d. $n\times n$ random stochastic matrices with elements of value $0$ or $1$ such that $\E[Z_t \mid I_t = u] = e_u a_u^\top + \sum_{v\neq u} e_v e_v^\top$ where $e_w$ denotes the $n$-dimensional standard basis vector, with the $w$-th element equal to $1$ and other elements equal to $0$.

The discrete-time \emph{$\epsilon$-noisy voter model} is defined by $X_0\sim \mu$, and 
\begin{equation}
X_{t+1,u} \mid X_t \sim \mathrm{Ber}(f(a_u^\top X_t)) \hbox{ for } t\geq 0 \hbox{ and } u \in \{1,\ldots,n\}
\label{equ:nv}
\end{equation}
where $f$ is some given function $f:[0,1]\rightarrow [\epsilon, 1-\epsilon]$, and $\epsilon \in [0,1/2]$. Under certain conditions on $f$ and $0 < \epsilon < 1/2$, the $\epsilon$-noisy voter model is an ergodic stochastic process. This is important for statistical inference of the model parameters as they can be inferred from a single, sufficiently long random realization of the stochastic process. This is in contrast to the voter model which requires several realizations of the voter model process for inference of the model parameters. 

A special case of an $\epsilon$-noisy voter model is \emph{the linear $\epsilon$-noisy voter model} defined by taking $f(x) = \epsilon + (1-2\epsilon)x$. Note that the linear $\epsilon$-noisy voter model corresponds to the voter model when $\epsilon = 0$. For the linear $\epsilon$-noisy voter model, we have
\begin{equation}
X_{t+1} = D(Q_{t+1})Z_{t+1}X_t + R_{t+1}
\label{equ:xqzxr}
\end{equation}
where $(Q_{t},R_{t})$ is an i.i.d. sequence of $n$ dimensional vectors with independent elements with distribution $\P[(Q_{t,u},R_{t,u}) = (q,r)] = p(q,r)$ with $p(0,0) = p(0,1)= \epsilon$ and $p(1,0) = 1-2\epsilon$, for all $u \in \{1,\ldots,n\}$, and $D(x)$ denoting diagonal matrix with diagonal elements $x$. The random linear dynamical system (\ref{equ:xqzxr}) can be seen as a randomly perturbed version of the random linear dynamical system (\ref{equ:xzx}). The role of this perturbation is significant, making an absorbing Markov chain to an ergodic Markov chain. 

Statistical inference for the voter model process is a challenging task because  the number of informative node interactions for parameter estimation 
vanish as node states converge to a consensus state. For a node interaction to be informative for the estimation, it is necessary that the node has neighbors with mixed states---having some neighbors in state $0$ and some in state $1$. For example, consider asynchronous discrete-time voter model, where at each time step a single, random node observes the state of a randomly picked neighbor and updates its state, with node interactions restricted to a path connecting $n$ nodes. Assume that initially $k$ nodes on one end of the path are in state $1$ and other nodes are in state $0$. Then, we can show that the expected number of nodes participating in at least one informative interaction until absorption to a consensus state is $k(\log(n/k)+\Theta(1))$ when $k = o(n)$. For the given voter model instance, typically, informative interactions will be observed only for a small fraction of nodes in each realization of the voter model process. In general, for the voter model inference, both the matrix of pairwise interaction rates $A$ and the initial state distribution $\mu$ play an important role. 

We present new results on statistical inference for the voter model. This is achieved by using a framework that allows to study inference for absorbing stochastic processes, by learning from several realizations of the underlying stochastic process. This is different from existing  work on statistical inference for stationary autoregressive stochastic processes, akin to the aforementioned $\epsilon$-noisy voter model. In order to study statistical inference for absorbing stochastic processes, we need certain properties of the hitting time of an absorbing state. Specifically, for the voter model, we need bounds on the expected value and a probability tail bound of consensus time. We present new results for the latter two properties of the consensus time, which may be of general interest. Before summarizing our contributions in some more detail, we review related work. 


\subsection{Related work} Prior work on voter models is mostly concerned with dynamics of voter processes on graphs, studying properties such as hitting probabilities and time to reach an absorbing state. Several seminal works studied voter model in continuous time, where interactions between vertices occur at events triggered by independent Poisson processes associated with vertices or edges, e.g. \cite{cox1989,L85,O12}. The discrete-time voter model was first studied in \cite{NIY99} and \cite{HP02} under assumption that $A$ is a stochastic matrix such that $a_{u,v} > 0$ if and only if $a_{v,u} > 0$ and the support of $A$ corresponds to the adjacency matrix of a nonbipartite graph $G$. \cite{HP02} found a precise characterization of the hitting probabilities of absorbing states, showing that $\lim_{t\rightarrow \infty}\P[X_t = \vec{1}] = 1-\lim_{t\rightarrow \infty}\P[X_t = \vec{0}] = \pi^\top x$, for any initial state $x$, where $\pi$ is the stationary distribution of $A$, i.e. a unique distribution $\pi$ satisfying $\pi^\top = \pi^\top A$. 

The consensus time of voter model process was studied in previous work under various assumptions. In an early work, \cite{cox1989} studied coalescing random walks and voter model consensus time on a torus. Most of works studied voter model process with node interactions defined by a graph $G = (V,E)$ where $V$ is the set of vertices and $E$ is the set of edges. In \cite{HP02}, using duality between voter model process and coalescing random walks, it was shown that the expected consensus time is $O(m(G) \log(n))$, where $m(G)$ is the worst-case expected meeting time of two random walks on graph $G$. \cite{kanade} showed that $m(G) = O((n d_{\max}/\Phi(G))\log(d_{\max}))$, for a lazy random walk, where $d_{\max}$ is the maximum node degree and $\Phi(G)$ is graph conductance. This lazy random walk, in each time step remains at the current vertex with probability $1/2$ and otherwise moves to a randomly chosen neighbor. Combined with the result in \cite{HP02}, this implies the expected consensus time bound $\tilde{O}((n d_{\max})/\Phi(G))$. \cite{petra} studied a voter model process where in each time step, every vertex copies the state of a randomly selected neighbor with probability $1/2$, and, otherwise, does not change its state (corresponding to the previously defined lazy random walk). In this setting, they showed that the expected consensus time is $O((d(V)/d_{\min})/\Phi(G))$, where $d_{\min}$ is the minimum node degree and $d(V)$ is the sum of node degrees. Our bound on the expected consensus time is competitive to the best previously known bound up to at most a logarithmic factor in $n$. \cite{AF02,CFR10} showed that if the states of vertices are initially distinct, the voter model process takes $\Theta(n)$ expected steps to reach consensus on many classes of expander graphs with $n$ vertices. \cite{CEOR13} established bounds on the expected consensus time for the voter model process on general connected graphs that depend on an eigenvalue gap of the transition matrix of random walk on the graph and the variance of the degree sequence. \cite{CR16} showed that the expected consensus time is $O(1/\Psi_A)$ where $\Psi_A$ is a property of $A$ (we discuss in Section~\ref{sec:examples}). \cite{roberto} established various results on hitting times of lazy random walks on graphs. Unlike to the aforementioned previous work, which found bounds on the expected consensus bound or bounds that hold with a constant probability, our results allow us to derive high probability bounds.

The problem of inferring node interaction parameters from observed node states was studied in an early work by \cite{NS12} for classic epidemic models, and more recently by \cite{PH15} for independent cascade model and some stationary voter model processes, as well as by \cite{Diff16} for some continuous-time network diffusion processes. None of these works considered statistical inference for an absorbing voter model process. A recent line of work studied statistical estimation for sparse autoregressive processes, including vector autoregressive processes \cite{basu2015,HRW19,HRW16,NVR17,ZP18}, sparse Bernoulli autoregressive processes \cite{SparseBer2019,MRW19b,KBS19}, and network Poisson processes \cite{MRW19}. These works established convergence rates for the parameter estimation problem. All these works are concerned with stationary autoregressive processes and thus do not apply to inference of absorbing stochastic processes, such as the voter model we study. Our work provides a framework to study statistical inference for absorbing stochastic processes, which may be of independent interest for autoregressive stochastic processes. Another related work is on identification of discrete-time linear dynamical systems with random noise, e.g. recent works by \cite{SMTJR18} and \cite{JP19}. We present a new lower bound on the sampling complexity for estimation of voter model parameters by studying the random linear dynamical system that governs the evolution of the voter model process.

\subsection{Summary of our contributions} We show an upper bound on the expected consensus time
$$
\EP[\tau]\leq \frac{1}{\Phi_A}\log\left(\frac{1}{2\pi^*}\right) 
$$
where $\pi^* = \min\{\pi_v: v = 1,\ldots,n\}$, $\pi$ is the stationary distribution of $A$, and $\Phi_A$ is a parameter of $A$. The upper bound is tight in the sense that there exist voter model instances that have the expected consensus time matching the upper bound up to a poly-logarithmic factor in $n$. The upper bound is obtained by a Lyapunov function analysis and follows from the exponential moment bound $\E[e^{\theta \tau}]\leq 1/(2\pi^*)$ that holds for any $\theta$ such that $(1-\Phi_A)e^{\theta}\leq 1$. This exponential moment bound allows us to bound the consensus time with high probability. Specifically, for any $\delta \in (0,1]$, $\tau \leq (\log(1/(2\pi^*)) + \log(1/\delta))/\Phi_A$ with probability at least $1-\delta$. In particular, this implies a high probability bound, $\tau \leq (\log(1/(2\pi^*))+\log(n))/\Phi_A$, which holds with probability at least $1-1/n$. Moreover, by using the aforementioned exponential moment bound, we can bound any moment of consensus time. These results are instrumental for the voter model inference problem, and may also be of independent interest. 

%

We developed a methodology for establishing statistical estimation error bounds for absorbing stochastic processes. This is based on using the framework for analysis of $M$-estimators with decomposable regularizers from high-dimensional statistics along with bounds on the expected length of cycles and probability tail bounds for the length of cycles. To obtain these results, we leverage our bounds on the expected consensus time and probability tail bounds for the consensus time, and use probability tail bounds for some super-martingale sequences. 

We show that the parameter estimation error of a voter model with parameter $A$ due to statistical estimation errors, measured by squared Frobenious norm, is 
$$
\tilde{O}\left(\frac{s}{\alpha^2 (1/\Phi_{A})\lambda_{\min}(\E[X_0 X_0^\top])^2}\frac{1}{m}\right)
$$  
with high probability, for sufficiently large number $m$. Here $s$ is an upper bound on the support of the voter model parameter $A$, $\lambda_{\mathrm{min}}(\E[X_0X_0^\top])$ is the smallest eigenvalue of the correlation matrix of the extended voter process with respect to stationary distribution, and $\alpha > 0$ is a lower bound for any non-zero element of $A$.  

We also present a lower bound on the sampling complexity for statistical inference of the voter model parameters, using the framework of locally stable estimators. Roughly speaking, for the voter model with parameter $A$ with every element in its support of value at least $\alpha > 0$, to have the Frobenious norm of the parameter estimation error bounded by $\epsilon$, with probability at least $1-\delta$, the number of voter model process realizations, $m$, must satisfy, for every sufficiently small $\epsilon > 0$,
$$
m  \geq \frac{\alpha}{16}\frac{1}{\epsilon^2 \EP[\tau] \lambda_{\min}(\E\left[X_0 X_0^\top\right])}\log\left(\frac{1}{2.4\delta}\right).
$$

\subsection{Organization of the paper}

In Section~\ref{sec:model} we provide additional definitions for model formulation and some mathematical background for the analysis in the paper. Section~\ref{sec:consensus} contains our main results on the consensus time of the voter model process. Specifically, this includes the exponential moment bound in Theorem~\ref{lm:exp} from which we derive a bound on the expected consensus time in Corollary~\ref{cor:etau}, a bound on any moment of consensus time in Corollary~\ref{cor:mom}, and a probability tail bound for a sum of independent consensus times in Theorem~\ref{thm:tail}. Section~\ref{sec:estimation} contains our results on statistical estimation, with an upper bound provided in Theorem~\ref{thm:paramerr} and a lower bound in Theorem~\ref{thm:lb}. The supplementary material contains missing proofs and some further results.  

\section{Preliminaries}
\label{sec:model}
In this section we provide additional details for the model formulation and then provide some background definitions and results that we use in the rest of the paper.

\subsection{Model formulation}

The voter model is defined as a Markov chain $\{X_t\}_{t\geq 0}$ on the state space $\mathcal{X}=\{0,1\}^n$ with initial state $X_0$ with distribution $\mu$ and the state transitions defined by 
\begin{equation}
X_{t+1} = Z_{t+1} X_t, \hbox{ for } t \geq 0
\label{equ:xt}
\end{equation}
where $Z_1, Z_2, \ldots$ is an i.i.d. random sequence of stochastic matrices taking values in $\{0,1\}^{n\times n}$ with independent rows. We can interpret $X_{t,u}$ as the state of vertex $u\in V:=\{1,\ldots,n\}$ at time $t$, which takes value $0$ or $1$.

The system (\ref{equ:xt}) is a time-variant linear dynamical system with random i.i.d. linear transformations as defined above. 

We use the notation 
$$
A = \E[Z_1].
$$

We assume that $A$ is an aperiodic, irreducible transition matrix. Under these assumptions, $A$ has a stationary distribution $\pi$ which is unique, given as the solution of global balance equations $\pi^\top = \pi^\top A$.  We denote with $a_u^\top$ the $u$-th row of matrix $A$. Note that $a_u$ can be interpreted as a probability distribution according to which vertex $u$ initiates pairwise interactions. Any two vertices $u$ and $v$ are said to be \emph{neighbors} if $a_{u,v} > 0$, i.e. if the two vertices interact with a positive probability.

For the voter model, there are two absorbing states $C = \{\vec{0},\vec{1}\}$, and all other states are transient. Let $C_0=\{\vec{0}\}$ and $C_1 = \{\vec{1}\}$, hence $C = C_0\cup C_1$. We refer to either of the two absorbing states as a consensus state. Let $\tau$ denote the hitting time of a consensus state, we refer to as the \emph{consensus time}, which is defined by
$$
\tau = \min\{t\geq 1: X_t \in C\}.
$$

In our analysis, we consider the extended voter process $\{X_t\}_{t\in \mathbb{Z}}$ defined by cycles of individual voter model process realizations. Let $\{T_i\}_{i\in \mathbb{Z}}$ be a point process defined as follows. Let $T_i$ be the time at which the $i$-th voter model process is in its initial state, and let $S_i = T_{i+1}-T_i$. We assume that $X_t$ for $T_i\leq t < T_{i+1}$, correspond to the states of the $i$-th voter model process until reaching a consensus state, excluding its final consensus state. Note that, indeed, $S_i = \tau_i$, where $\tau_i$ is the consensus time of the $i$-th voter model process. In some parts of our analysis, we will also consider the alternative definition of the extended voter process which includes the final states of individual voter model processes. In this case, $S_i = \tau_i + 1$. The stationary distributions of the two extended voter processes are different.

We denote with $\PP[A]$ the probability of an event $A$ conditional on time $0$ being a point, i.e. $\{T_0=0\}$. We denote with $\P[A]$ the probability of event $A$ under stationary distribution, where $0$ is an arbitrary time. In the framework of stationary point processes, the two distributions are referred to as the Palm distribution and stationary distribution, respectively. By the Palm inversion formula, for any measurable function $f:\mathcal{X}\mapsto \reals$,
\begin{equation}
\E[f(X_0)] = \frac{\EP\left[\sum_{t=0}^{S_1-1}f(X_t)\right]}{\EP[S_1]}.
\label{equ:palm}
\end{equation}



We will also use the following notation in our analysis of parameter estimation. Let $\{X_t^{(i)}\}_{t\geq 0}$, $i = 1,\ldots, m$ be independent voter model processes, each with initial state distribution $\mu$. Let $\tau_i$ denote the consensus time of voter model process $i$. With a slight abuse of notation, we will sometimes write $X_t$ in lieu of $X_t^{(1)}$ and $\tau$ in lieu of $\tau_1$.

\subsection{Markov chains background} In this section we present some definitions and results from Markov chain theory that we use in our analysis.

Let $\{X_t\}_{t\geq 0}$ be a time homogeneous Markov chain on a state space $(\mathcal{X},\mathcal{E})$. Let $P(x,A)$, $x\in \mathcal{X}$, $A\in \mathcal{E}$ denote the transition probability and let $P$ denote the corresponding operators on measurable functions mapping $\mathcal{X}$ to $\reals$. Let $P_t(x,\cdot)$ denote the transition probabilities at time $t$. We define the following conditions:

\begin{description}
\item[(A1)] \emph{Minorization condition}. There exist $S\in \mathcal{E}$, $\epsilon > 0$, and a probability measure $\nu$ on $(\mathcal{X},\mathcal{E})$ such that
$$
P(x,A)\geq \epsilon \nu(A)
$$
for all $x\in S$ and $A\in \mathcal{E}$.
\item[(A2)] \emph{Drift condition}. There exist a  measurable function $V:\mathcal{X}\rightarrow [1,\infty)$ and constants $\lambda < 1$ and $K<\infty$ satisfying
$$
PV(x) := \E[V(X_1)\mid X_0 = x] \leq \left\{
\begin{array}{ll}
\lambda V(x) & \hbox{ if } x\notin S\\
K & \hbox{ if } x\in S.
\end{array}
\right . 
$$
\item[(A3)] \emph{Strong aperiodicity condition}. There exists $\tilde{\epsilon} > 0$ such that $\epsilon \nu(S)\geq \tilde{\epsilon}$. 
\end{description}

We say that a measurable function $V:\mathcal{X}\rightarrow [1,\infty)$ is a \emph{drift function} for $P$ with respect to $S$, with constants $\lambda < 1$ and $K <\infty$, if it satisfies (A2).

For any given set $S\subset \mathcal{E}$, let us define the hitting time
$$
\tau_S = \min\{t > 0: X_t \in S\}.
$$

We say that the set $S$ is an \emph{atom} if $P(x,\cdot)=P(y,\cdot)$ for all $x,y\in S$. In this case, we may assume $\epsilon = 1$ and $P(x,\cdot)=\nu(\cdot)$ for all $x\in S$. 


By Theorem~1.1 in \cite{B05}, under (A1)-(A3), $\{X_t\}_{t\geq 0}$ has a unique stationary distribution $\pi$ and $\E_{X\sim \pi}[V(X)]<\infty$. Moreover, there exists $\rho < 1$ depending only on $\epsilon$, $\tilde{\epsilon}$, $\lambda$ and $K$ such that whenever $\rho < \gamma < 1$, there exists $M < \infty$ depending only on $\gamma$, $\epsilon$, $\tilde{\epsilon}$, and $K$ such that
$$
\sup_{|g|\leq V}\left|\E_x[g(X_t)]-\E_{\pi}[g(X_0)]\right|\leq M V(x)\gamma^t
$$
for all $x\in \mathcal{X}$ and $t\geq 0$, where the supremum is over all measurable functions $g:\mathcal{X}\rightarrow \reals$ satisfying $|g(x)|\leq V(x)$ for all $x\in \mathcal{X}$. If $g$ is restricted to functions satisfying $|g(x)|\leq 1$ for all $x\in \mathcal{X}$, then we have the standard geometric ergodicity condition
$
||P_t(x,\cdot)-\pi||_{TV} \leq M V(x) \gamma^t
$
where $||\cdot||_{TV}$ denotes the total variation distance. The Markov chain is said to be $(M,\rho)$-\emph{geometrically ergodic} if
$
||P_t(x,\cdot)-\pi||_{TV} \leq M \rho^t
$
for some $M < \infty$ and $\rho < 1$.

The following is a key lemma (e.g. Lemma 2.2 and Theorem 3.1 \cite{LT96}, Proposition~4.1 \cite{B05}) that we will use in our analysis of consensus time.

\begin{lm} Let $\{X_t\}_{t\geq 0}$ be a Markov chain on $(\mathcal{X},\mathcal{E})$ with transition kernel $P$, and let $C\in \mathcal{E}$. Suppose that $V:\mathcal{X}\rightarrow [1,\infty)$ is a measurable function that satisfies $PV(x)\leq \lambda V(x)$ for all $x\notin C$, for a fixed $\lambda < 1$. Then, for all $x\in \mathcal{X}$,
$$
\E_x[\lambda^{-\tau_C}]\leq V(x).
$$
\label{lm:drift}  
\end{lm}

The lemma can be established by some Lyapunov drift arguments. 

\subsection{Miscellaneous definitions} We use different definitions of norms. For every $x\in \reals^n$, $||x||_{p}$ denotes the $L_p$ norm, $||x||_p = (|x_1|^p + \cdots + |x_n|^p)^{1/p}$. For every matrix $X = (x_1, \ldots, x_n)^\top \in \reals^{n\times m}$, $||X||_{p,q}$ is defined as
$$
||X||_{p,q} = (||x_1||_q^p + \ldots + ||x_n||_q^p)^{1/p}.
$$
In particular, $||X||_{1,1}$ is the sum of absolute values of elements of $X$. $||X||_0$ denotes the number of non-zero elements in $X$, i.e. the \emph{support} of $X$. The \emph{Frobenius norm} $||X||_F$ is defined as 
$$
||X||_F = ||X||_{2,2} = \sqrt{\sum_{i=1}^n \sum_{j=1}^m x_{i,j}^2}.
$$

\section{Consensus time} 
\label{sec:consensus}
In this section we show our results on consensus time of the voter model process. 

For any vector $a\in \reals^n$, let us define
$$
V_a(x) = a^\top x (1-a^\top x).
$$ 

For the voter model with parameter $A$, $V_{a_u}(X_t)$ is the variance of $X_{t+1,u}$ conditional on $X_t$. Intuitively, we may interpret $V_{a_u}(X_t)$ as a measure of diversity of states of neighbors of vertex $u$. Note that $V_{a_u}(X_t)=0$ if and only if all neighbors of vertex $u$ are in the same state (either $0$ or $1$). 

\begin{lm} For every $x \in \{0,1\}^n$ and $t\in \mathbb{Z}$, function $V_\pi$ satisfies the following expected drift equation:
$$
\E[V_\pi(X_{t+1})- V_\pi(X_t) \mid X_t = x] \\
=  - \sum_{u=1}^n \pi_u^2 V_{a_u}(x) + \EP[V_\pi(X_0)]\ind_{\{x\in C\}}.
$$
\label{lem:drift}
\end{lm}

We can interpret the term $\sum_{u=1}^n \pi_u^2 V_{a_u}(x)$ as a weighted sum of variances of Bernoulli distributions with parameters $a_u^\top x$, associated with vertex neighborhood sets. The weights are equal to the squares of the elements of the stationary distribution $\pi$. By taking expectation on both sides in the equation in Lemma~\ref{lem:drift} with respect to the stationary distribution, under assumption that $\mu(C) = 0$, we have
\begin{equation}
(\EP[\tau] + 1) \sum_{u=1}^n \pi_u^2 \E[V_{a_u}(X_0)] = \EP[V_\pi(X_0)].
\label{equ:etaue}
\end{equation}
From (\ref{equ:etaue}), we can observe that the expected consensus time is fully determined by the expected variance of vertex states with respect to the stationary distribution measured by $\sum_{u=1}^n \pi_u^2 \E[V_{a_u}(X_0)]$ and the variance of the initial state measured by $\EP[V_\pi(X_0)]$. Intuitively, the smaller the value of $\sum_{u=1}^n \pi_u^2 \E[V_{a_u}(X_0)]$, the larger the expected consensus time. 

The following is a key property of $A$ in our analysis of consensus time 
\begin{equation}
\Phi_A = \min\left\{\frac{\sum_{u=1}^n \pi_u^2 V_{a_u}(x)}{\pi^\top x(1-\pi^\top x)} : x\in \{0,1\}^n, x\notin C\right\}.
\label{equ:phia}
\end{equation}

Note that $0 < \Phi_A \leq 1$. The inequality $\Phi_A >0$ can be shown by contradiction as follows. Suppose $\Phi_A = 0$, which is equivalent to $a_u^\top x (1-a_u^\top x) = 0$ for all $u\in V$. We can then partition $V$ into two non-empty sets $S$ and $V\setminus S$ such that each $u$ has support of $a_u$ fully contained in either $S$ or $V\setminus S$. This implies that $A$ has a block structure, which contradicts the assumption that $A$ is irreducible and aperiodic. The inequality $\Phi_A \leq 1$ follows from
\begin{eqnarray*}
\sum_{u=1}^n \pi_u^2 a_u^\top x(1-a_u^\top x) 
& \leq & \sum_{u=1}^n \pi_u a_u^\top x(1-a_u^\top x)\\
& \leq & \left(\sum_{u=1}^n \pi_u a_u^\top x\right)\left(1-\left(\sum_{u=1}^n \pi_u a_u^\top x\right)^2\right)\\
& = & \pi^\top x (1-\pi^\top x)
\end{eqnarray*}
where the first inequality is by non-negativity of the summation terms and $\pi_u \in [0,1]$ for all $u\in V$, the second inequality is by concavity of $x \mapsto x(1-x)$, and the equation is by the global balance equations $\pi^\top = \pi^\top A$.

By Lemma~\ref{lem:drift} and definition of $\Phi_A$, we have the following corollary.

\begin{cor} For all $x \in \{0,1\}^n$ and $t\in \mathbb{Z}$,
$$
\E[V_\pi(X_{t+1}) - V_\pi(X_t) \mid X_t = x] \\
\leq   - \Phi_A V_\pi(x) + \EP[V_\pi(X_0)]\ind_{\{x\in C\}}.
$$
\label{cor:drift}
\end{cor}

We next present a bound on the exponential moment of consensus time.

\begin{thm} For any $x\in \{0,1\}^n$ such that $x\notin C$, and any $\theta\in \reals$ such that $(1-\Phi_A)e^{\theta}\leq 1$,
$$
\EP_x[e^{\theta\tau}] \leq \frac{V_\pi(x)}{\min_{z\in \{0,1\}^n\setminus C}V_\pi(z)}.
$$
\label{lm:exp}
\end{thm}

\begin{proof} The theorem follows from the general result for Markov chains satisfying the Lyapunov drift condition stated in Lemma~\ref{lm:drift}. Recall that $C = \{\vec{0},\vec{1}\}$. Let $V(x) := V_\pi(x)/\min_{z\in \{0,1\}^n\setminus C}V_\pi(z)$. Using Corollary~\ref{cor:drift}, $V(x)$ is a drift function with respect to $C$, with constants $\lambda = 1-\Phi_A$ and $K = \EP[V(X_0)]$. By Lemma~\ref{lm:drift}, we have $\E_x[(1-\Phi_A)^{-\tau}]\leq V(x)$. Combining with the condition $(1-\Phi_A)e^{\theta}\leq 1$, the claim of the theorem follows. 
\end{proof}

We have the following upper bound for the expected consensus time.  

\begin{cor} For every $x\in \{0,1\}^n$ such that $x\notin C$,
$$
\EP_x[\tau] \leq \frac{1}{\Phi_A}\log\left(\frac{1}{2\pi^*}\right)
$$
where $\pi^* = \min\{\pi_v: v = 1,\ldots, n\}$.
\label{cor:etau}
\end{cor}

\begin{proof} By Theorem~\ref{lm:exp}, taking $\theta$ such that $(1-\Phi_A)e^{\theta} = 1$, and Jensen's inequality, we have
$$
\EP_x[\tau] \leq \frac{1}{\log\left(\frac{1}{1-\Phi_A}\right)}\log\left(\frac{V_\pi(x)}{\min_{z\in \{0,1\}^n\setminus C}V_\pi(z)}\right).
$$
The corollary follows from the last inequality and combining with the following facts (a) $V_\pi(x) \leq 1/4$ for all $x \in \{0,1\}^n$, (b) $V_\pi(z) \geq \pi^*(1-\pi^*)\geq \pi^*/2$, for all $z\in \{0,1\}^n\setminus C$, and (c) $1/\log(1/(1-\Phi_A)) \leq 1/\Phi_A$.
\end{proof}

From (\ref{equ:etaue}), we obtain
$$
\EP[\tau] \geq \frac{4\EP[V_\pi(X_0)]}{||\pi||_2^2} - 1.
$$
By Corollary~\ref{cor:etau}, $\EP[\tau]\leq \log(1/(2\pi^*))/\Phi_A$. Hence, if $\Phi_A = \Omega(||\pi||_2^2)$ and $\EP[V_\pi(X_0)]=\Omega(1)$, the upper bound is tight within a factor logarithmic in $1/\pi^*$. For instance, for the complete graph case, i.e. when $a_{u,u} = 0$ and $a_{u,v} = 1/(n-1)$ for all $u\neq v$, we have $\Phi_A = (1/n)(1+o(1))$ (we show this in Section~\ref{sec:examples}). In this case, we have the upper bound $\EP[\tau] = O(n\log(n))$, which is within a factor logarithmic in $n$ of the lower bound $\EP[\tau] = \Omega(n)$.  

In fact, from Theorem~\ref{lm:exp}, we have the following bound for any moment of the consensus time.

\begin{cor} For every $x\in \{0,1\}^n$ such that $x\notin C$ and $k\geq 0$,
$$
\EP_x[\tau^k] \leq \frac{1}{2}\left(\frac{k}{e}\right)^k \frac{1}{\Phi_A^k}\frac{1}{\pi^*}.
$$
\label{cor:mom}
\end{cor}

\begin{proof} The proof follows readily from Theorem~\ref{lm:exp} and the the elementary fact that for any non-negative random variable $X$, for any $k\geq 0$ and $\theta > 0$, $\E[X^k]\leq (k/(e\theta))^k\E[e^{\theta X}]$.
\end{proof}

The bound for the first moment of consensus time in Corollary~\ref{cor:etau} is better than that in Corollary~\ref{cor:mom} in having a logarithmic dependence on $1/\pi^*$ instead of linear dependence on this parameter. 

By using the bound on the exponential moment of consensus time in Theorem~\ref{lm:exp}, we can obtain a bound on the tail probability of the consensus time. This allows us to derive bounds on the consensus time that hold with high probability. We will next state a more general result that applies to the sum of consensus times of $m\geq 1$ independent voter model processes. We will use this more general result for the parameter estimation in Section~\ref{sec:estimation}.

\begin{thm} For $m\geq 1$ independent voter model processes with parameter $A$ and independent initial states according to distribution $\mu$, for any $a\geq 0$,
\begin{equation}
\PP\left[\sum_{i=1}^m \tau_i \geq m a\right] \leq \left(\left( \frac{\EP[V_\pi(X_0)]}{\min_{z\in \{0,1\}^n\setminus C} V_\pi(z)}\right)(1-\Phi_A)^a\right)^{m}.
\label{equ:tail}
\end{equation}
\label{thm:tail}
\end{thm}

\begin{proof} By Chernoff's bound, for any $\theta\geq 0$,
\begin{eqnarray*}
\PP\left[\sum_{i=1}^m \tau_i \geq m a\right] & \leq & e^{-ma \theta}\EP\left[e^{\theta\sum_{i=1}^m \tau_i}\right]\\
& = & e^{-ma \theta} \EP[e^{\theta\tau_1}]^m.
\end{eqnarray*}

Let $\theta^* = -\log(1-\Phi_A)$. By Theorem~\ref{lm:exp}, we have
$$
\EP[e^{\theta^*\tau_1}] \leq \frac{\EP[V_\pi(X_0)]}{\min_{z\in \{0,1\}^n\setminus C} V_\pi(z)}.
$$
Hence, (\ref{equ:tail}) follows.
\end{proof}

From Theorem~\ref{thm:tail}, for any $\delta \in (0,1]$, with probability at least $1-\delta$, 
$$
\frac{1}{m}\sum_{i=1}^m \tau_i \leq \left(\log\left(\frac{\EP[V_\pi(X_0)]}{\min_{z\in \{0,1\}^n\setminus C} V_\pi(z)}\right) + \frac{1}{m}\log\left(\frac{1}{\delta}\right)\right)\frac{1}{\log\left(\frac{1}{1-\Phi_A}\right)}.
$$

From the last statement, we have the following corollary.

\begin{cor} For $m\geq 1$ independent voter model processes with parameter $A$ and independent initial states with distribution $\mu$, for any $\delta\in (0,1]$, with probability at least $1-\delta$,
$$
\frac{1}{m}\sum_{i=1}^m \tau_i \leq \frac{1}{\Phi_A}\left(\log\left(\frac{1}{2\pi^*}\right)+\frac{1}{m}\log\left(\frac{1}{\delta}\right)\right).
$$
\label{cor:sumtau}
\end{cor}

The corollary implies us a high probability bound for consensus time, $O((1/\Phi_A)(\log(1/\pi^*)+\log(n)))$, which holds with probability at least $1-1/n^c$, for any constant $c > 0$.

\paragraph{Asynchronous discrete-time voter model} Similar results hold for the asynchronous discrete-time voter model. An analogous lemma to Lemma~\ref{lm:drift} holds which is given as follows.

\begin{lm} For every $x\in \{0,1\}^n$ and $t\in \mathbb{Z}$, function $V_\pi$ satisfies the following expected drift equation:
\begin{eqnarray*}
&& \E[V_\pi(X_{t+1})-V_\pi(X_t) \mid X_t = x] \\
&=& - \frac{1}{n} \sum_{u=1}^n \pi_u^2 (x_u(1-a_u^\top x) + (1-x_u) a_u^\top x) + \EP[V_\pi(X_0)]\ind_{\{x\in C\}}.
\end{eqnarray*}
\label{lm:drift-async}
\end{lm}

The statements in Corollary~\ref{cor:etau} and Theorem~\ref{thm:tail} hold true for asynchronous discrete-time voter model by replacing $\Phi_A$ with $\Phi_A'$ where
\begin{equation}
\Phi_A' = \frac{1}{n}\min\left\{\frac{\sum_{u=1}^n \pi_u^2 V_{u}'(x)}{\pi^\top x(1-\pi^\top x)}: x\in \{0,1\}^n, x\neq C\right\}
\label{equ:phip}
\end{equation}
with $V'_u(x) := x_u(1-a_u^\top x) + (1-x_u)a_u^\top x$.

It is readily observed that $n\Phi_A' \geq \Phi_A$ and it follows that the statements in Corollary~\ref{cor:etau} and Theorem~\ref{thm:tail} remain to hold true for asynchronous discrete-time voter model by replacing $\Phi_A$ with $\Phi_A/n$. The factor $1/n$ occurs because under asynchronous discrete-time voter model, at each time step, exactly one node updates its state, which is in contrast to the voter model under which all nodes update their states.  

\subsection{Discussion and comparison with previously-known consensus time bounds}
\label{sec:examples}
We discuss the value of parameter $\Phi_A$ and expected consensus time for some node interaction matrices $A$ and compare with the best previously-known consensus time bounds. 

We will discuss node interaction probabilities that can be defined by a graph, as common in the literature on voter model processes and random walks on graphs. Let $G = (V, E)$ be a connected graph where $V$ is the set of $|V|=n$ vertices and $E$ is the set of edges. Let $d_v$ denote the \emph{degree} of vertex $v$, defined as the number edges incident to vertex $v$. For any set $S\subseteq V$, let $d(S) = \sum_{v\in S}d_v$, and let $d_v(S)$ denote the number edges incident to $v$ and $S$. Let $d_{\min}$ denote the minimum degree of a vertex in $G$. Graph $G$ may contain self-loops, i.e. an edge connecting a vertex with itself. 

We represent any given vector $x\in \{0,1\}^n$ by the set $S=\{v\in V: x_v = 1\}$, and we will use the notation $S^c := V\setminus S$. Note that a vector $x\in \{0,1\}^n$ defines a graph partition, i.e. partition of the set of vertices into two components, $S$ and $S^c$.

We will compare $\Phi_A$ and bonds on the expected consensus time with some functions of graph conductance. Conductance $\Phi(G)$ of graph $G$ is defined as
\begin{equation}
\Phi(G) = \min_{S \subset V: 0 < |S| < n} \frac{|E(S,S^c)|}{\min\{d(S),d(S^c)\}}
\label{equ:psistar}
\end{equation}
where $E(S,S^c)$ is the set of edges connecting $S$ and $S^c$. Let $L$ be the normalized Laplacian matrix of graph $G$, defined as $L =I -  D^{-1/2}AD^{-1/2}$ where $D$ is the diagonal matrix with diagonal elements corresponding to vertex degrees.  Let $\lambda_2$ be the second smallest eigenvalue of $L$. By Cheeger’s inequality, for any connected graph $G$, 
\begin{equation}
\lambda_2/2 \leq \Phi(G)\leq \sqrt{2\lambda_2}.
\label{equ:cheg}
\end{equation}

We first consider node interactions such that in each time step, every vertex copies the state of a randomly chosen neighbor with probability $1/2$, where graph $G$ has no self-loops. For example, this case was studied in \cite{petra} and \cite{kanade}, and is commonly referred to as lazy random walk. The node interaction matrix $A$ has elements 
\begin{equation}
a_{u,v} = \frac{1}{2}\ind_{\{u=v\}} + \frac{1}{2}\frac{1}{d_u}\ind_{\{(u,v)\in E\}}, \hbox{ for } u,v\in V.
\label{equ:Abias}
\end{equation}

It can be readily checked that $\pi_v = d_v/d(V)$ for $v\in V$, and we also have
$$
\Phi_A = \frac{1}{4}\min_{S\subset V: 0 < |S| < n} \frac{\sum_{u\in V} (d_u \ind_{\{u\in S\}} + d_u(S))(d_u \ind_{\{u\in S^c\}}+d_u(S^c))}{d(S)d(S^c)}.
$$

\begin{lm} Assume that node interaction matrix $A$ is according to (\ref{equ:Abias}). Then, we have
$$
\frac{1}{\Phi_A} \leq 2 \frac{d(V)}{d_{\min}}\frac{1}{\Phi(G)}.
$$
\label{lem:phi-cond-bias}
\end{lm}

Together with Corollary~\ref{cor:etau}, Lemma~\ref{lem:phi-cond-bias} implies $\EP[\tau]=O(((d(V)/d_{\min})/\Phi(G))\log(n))$, which is within a logarithmic factor in $n$ to the expected consensus time bound in \cite{petra}.

For comparing with the expected consensus time bound in \cite{CR16}, $\EP[\tau]\leq 64/\Psi_A$, we consider $\Psi_A$ which is defined as $\Psi_A = \pi^*\tilde{\Psi}_A$ with
$$
\tilde{\Psi}_A = \min_{x\in \{0,1\}^n\setminus C}\frac{\E\left[\left|\sum_{u=1}^n\pi_u \left(x_u -\sum_{v=1}^n Z_{u,v}x_v\right)\right|\right]}{\min\{\pi^\top x, 1-\pi^\top x\}}.
$$

\begin{lm} For node interaction matrix $A$ according to (\ref{equ:Abias}), we have
$$
\tilde{\Psi}_A \leq \Phi(G).
$$
\label{lem:psi-cond}
\end{lm}

From Lemmas \ref{lem:phi-cond-bias} and \ref{lem:psi-cond}, we have 
\begin{lm} For node interaction matrix $A$ according to (\ref{equ:Abias}), we have
$$
\frac{1}{\Phi_A} \leq 2\frac{1}{\Psi_A}.
$$
\end{lm}

The last lemma, together with Corollary~\ref{cor:etau}, implies that our bound on the expected consensus time is at most a logarithmic factor in $n$ to the expected consensus time bound in \cite{CR16}.

We also considered another type of node interactions, where in each time step, each node copies the state of a node from its neighborhood set which includes the node itself. For space reasons, this discussion is deferred to Appendix~\ref{sec:fd}.

We next provide explicit characterizations of $\Phi_A$ and bounds on the expected consensus time, from Corollary~\ref{cor:etau}, for node interactions such that $a_{u,v} = 1/d_v$ for $(u,v)\in E$, for the case of a complete graph and a cycle. Note that it holds
\begin{equation}
\Phi_A = \min_{S\subset V: 0<|S|<n} \frac{|E_2(S,S^c)|}{d(S)d(S^c)}
\label{equ:phigraph0}
\end{equation}
where $E_2(S,S^c)$ is the set of paths with two edges connecting  $S$ and $S^c$.

\paragraph{Complete graph $K_n$} Let $G$ be the complete graph with $n$ vertices. Then, for any $S\subseteq V$, 
\begin{eqnarray*}
|E_2(S,S^c)| &=& |S|(n-|S|)(n-2)\\
d(S) &=& |S|(n-1)\\
d(S^c) &=& (n-|S|)(n-1).
\end{eqnarray*}

Using this in (\ref{equ:phigraph0}), we have
$$
\Phi_A = \frac{n-2}{(n-1)^2} = \frac{1}{n}(1+o(1)).
$$
Combining with $\pi^* = 1/n$, and Corollary~\ref{cor:etau}, we have
$$
\EP[\tau]\leq n \log(n)(1+o(1)).
$$

\paragraph{Cycle $C_n$} Let $G$ be the cycle with $n \geq 3$ nodes. Then, for any $S\subseteq V$, we have
$$
d(S) = 2|S| \hbox{ and } d(S^c) = 2(n-|S|).
$$
Conditional on $|S|$, the smallest value of $|E_2(S,S^c)|$ is achieved when vertices in $S$ are adjacent. In this case, we distinguish two cases. First, if $|S| = n-1$, then, $|E_2(S,S^c)| = 2$, and $|E_2(S,S^c)|/(d(S)d(S^c)) = 1/(2(n-1))$. Second, if $|S| < n-1$, then $|E_2(S,S^c)| = 4$, and thus
$$
\frac{|E_2(S,S^c)|}{d(S)d(S^c)} = \frac{1}{|S|(n-|S|)}.
$$
The minimum over $|S|$ is achieved for $|S| = n/2$ if $n$ is even, otherwise, it is achieved for $(n-1)/2$. For $n$ even, $\Phi_A = 4/n^2$ and, otherwise, $\Phi_A = 4/(n^2-1)$. Therefore, we have
$$
\Phi_A = 4\frac{1}{n^2}(1+o(1)).
$$

Combining with $\pi^*=1/n$, and Corollary~\ref{cor:etau}, we have
$$
\EP[\tau] \leq \frac{1}{4}n^2 \log(n)(1+o(1)).
$$

\section{Parameter estimation}
\label{sec:estimation}
In this section we first show an upper bound for the parameter estimation error, using a maximum likelihood estimator with a regulaizer, for the voter model parameter $A^*$ from observed node states over time for $m\geq 1$ independent voter model process realizations with parameter $A^*$ and independent initial states according to distribution $\mu$. We then show a lower bound for the parameter estimation error for the class of locally stable estimators.  

\subsection{Parameter estimation error upper bound} 

Let $A\mapsto \mathcal{L}(A; X)$ denote a loss function of the voter model for given observation data $X$, where $X$ are observed node states for $m$ independent voter model process realizations with parameter $A^*$ and initial state distribution $\mu$, 
$$
X = (X_0^{(1)}, \ldots, X_{\tau_1}^{(1)}, \ldots, X_0^{(m)}, \ldots, X_{\tau_m}^{(m)})^\top.
$$

We define estimator $\hat{A}$ as a minimizer of the loss function $\mathcal{L}(A; X)$, i.e.,
\begin{equation}
\hat{A}\in \arg\min_{A\in \Theta} \{\mathcal{L}(A; X)\}
\label{equ:mle}
\end{equation}
where $\Theta$ is some given set of parameters. 

Specifically, we will consider the loss function $\mathcal{L}(A; X)$ defined as the sum of the negative log-likelihood function and a regularizer defined as
\begin{equation}
\mathcal{L}(A; X) = -\ell(A; X) + \lambda_m ||A||_{1,1}
\label{equ:loss}
\end{equation}
where $\ell(A)$ is the log-likelihood function and $\lambda_m\geq 0$ is the regularization parameter. Let $\Delta = \hat{A}-A^*$ denote the parameter estimation error. As common in statistical inference theory, we will measure the parameter estimation error by the Frobenious norm $||\Delta||_F$.

The log-likelihood function can be expressed as
\begin{equation}
\ell(A; X) = \sum_{i=1}^m \left(\log(\mu(X_0^{(i)})) - \sum_{u=1}^n \sum_{t=0}^{\tau_i-1} H(X_{t+1,u}^{(i)}, a_u^\top X_t^{(i)})\right)
\label{equ:loglik}
\end{equation}
where $H(p,q)$ is the cross-entropy between two Bernoulli distributions with mean values $p$ and $q$, i.e.
$$
H(p,q) = -(p \log(q) + (1-p)\log(1-q)).
$$
Let us define  
$$
T_{u,i} = \{t\in \{0,\ldots,\tau_i\}:\ 0<a_u^\top X_t^{(i)} < 1\}.
$$
Intuitively, $T_{u,i}$ is the set of time steps at which the state of the $i$-th voter model process is such that vertex $u$ has a pair of neighbors in different states. Note that having such mixed neighborhood sets is necessary for the parameter estimation, as otherwise, no useful information can be gained from observed vertex states for the parameter estimation.

For our analysis we will consider the gradient vector $\nabla \ell(A; X)$ and the Hessian matrix $\nabla^2 \ell(A; X)$ of the log-likelihood function. The gradient vector $\nabla \ell(A; X)$ has elements given as follows
$$
\frac{\partial}{\partial a_{u,v}} \ell(A; X) = \sum_{i=1}^m \sum_{t\in T_{u,i}} \left(\frac{X_{t+1,u}^{(i)}}{a_u^\top X_t^{(i)}} - \frac{1-X_{t+1,u}^{(i)}}{1-a_u^\top X_t^{(i)}}\right)X_{t,v}^{(i)}.
\label{equ:gradloglik}
$$
The Hessian matrix $\nabla^2 \ell(A; X)$ has elements given as follows
\begin{equation}
\frac{\partial^2}{\partial a_{u,v}\partial a_{u,w}} \ell(A; X) = -\sum_{i=1}^m \sum_{t\in T_{u,i}} \varphi_u(X_t^{(i)}, X_{t+1}^{(i)}) X_{t,v}^{(i)}X_{t,w}^{(i)}
\label{equ:hessianloglik1}
\end{equation}
where
$$
\varphi_u(X,Y) = \frac{Y_u}{(a_u^\top X)^2} + \frac{1-Y_u}{(1-a_u^\top X)^2}
$$
and 
\begin{equation}
\frac{\partial^2}{\partial a_{u,v}\partial a_{u',w}}\ell(A; X) = 0, \hbox{ if } u\neq u'.
\label{equ:hessianloglik2}
\end{equation}

For bounding the parameter estimation error, we use the framework for analysis of M-estimators with decomposable regularizers from high-dimensional statistics, e.g. \cite{NRW12,W19}. The parameter estimator defined by (\ref{equ:mle}) with the loss function (\ref{equ:loss}) is an instance of an M-estimator with a decomposable regularizer. 

For any set $S\subseteq V^2$ and $A \in \reals^{n\times n}$, let $A_S$ be the $n\times n$ matrix with support restricted to $S$, i.e. $A_S$ is such that $(A_S)_{u,v} = a_{u,v}$ if $(u,v)\in S$ and $(A_S)_{u,v} = 0$ if $(u,v) \in S^c = V^2\setminus S$. For any set $S\subseteq V^2$, let
$$
\mathbb{C}(S; A^*) := \{\Delta: ||\Delta_{S^c}||_{1,1} \leq 3 ||\Delta_S||_{1,1} + 4 ||A^*_{S^c}||_{1,1} \}.
$$ 
When the support of $A^*$ is contained in $S$, we have $||A^*_{S^c}||_{1,1} = 0$. 

For a given positive integer $s$, let $S^*$ be a minimizer of $||A^*_{S^c}||_{1,1}$ over $S\subseteq V$ such that $|S|\leq s$. Let $\mathbb{C}^*:=\mathbb{C}(S^*; A^*)$.


For any differentiable loss function $\mathcal{L} : \reals^{n\times n}\mapsto \reals$, we define the \emph{first-order Taylor error} as
$$
\mathcal{E}(\Delta) = \mathcal{L}(A^*+\Delta) - \mathcal{L}(A^*)-\nabla \mathcal{L}(A^*)^\top vec(\Delta)
$$
where $vec(\Delta)$ denotes the vector defined by stacking the rows of matrix $\Delta$.

A key concept in the framework of $M$-estimators is that of restricted strongly convex functions which is defined as follows. 

\begin{df}{(Restricted Strong Convexity (RSC))} A loss function $\mathcal{L}$ satisfies restricted strong convexity relative to $A^*$ and $S\subseteq V^2$ with curvature $\kappa > 0$ and tolerance $\gamma^2$ if, for all $\Delta\in \mathbb{C}(S;A^*)$,
$$
\mathcal{E}(\Delta)\geq \kappa ||\Delta||_F^2 - \gamma^2.
$$
\label{def:rsc}
\end{df}

The following bound on the parameter estimation error follows from the framework of M-estimators with decomposable regularizers (e.g. Theorem~1 in \cite{NRW12}).

\begin{thm} Assume that the loss function $\mathcal{L}(A; X)$ in (\ref{equ:loss}) has the regularization parameter $\lambda_m$ such that
\begin{equation}
\lambda_m \geq 2||\nabla \ell(A^*)||_\infty
\label{equ:lambda}
\end{equation}
and, for some $S\subseteq V^2$, the negative log-likelihood function $-\ell(A;X)$ satisfies the RSC condition relative to $A^*$ and $S$ with curvature $\kappa > 0$ and tolerance $\gamma^2$. Then, we have
$$
||\hat{A}-A^*||_F^2 \leq 9 |S|\left(\frac{\lambda_m}{\kappa}\right)^2 + \left(2\gamma^2\frac{1}{m} + 4||A^*_{S^c}||_{1,1}\right)\frac{\lambda_m}{\kappa}.
$$
\label{thm:mest}
\end{thm}

To bound the parameter estimation error for the voter model, we need to show (1) that condition (\ref{equ:lambda}) holds for the extended voter model with a given probability and (2) that the negative log-likelihood function satisfies the RSC condition with a given probability. 

We first show a lemma that allows us to set the regularization parameter $\lambda_m$ such that condition in (\ref{equ:lambda}) holds with high probability.

\begin{lm} For any $\delta \in (0,1]$, and any $m\geq 1$ independent realizations of the voter model process with parameter $A^*$ and initial distribution $\mu$, with probability at least $1-\delta$, 
\begin{equation}
||\nabla \ell(A^*)||_\infty \leq \sqrt{2}\frac{1}{\alpha} 
\frac{1}{\sqrt{\Phi_{A^*}}}
\sqrt{m}
c_{n,\delta,\pi^*}(m)
\label{equ:infnorm}
\end{equation}
where
$$
c_{n,\delta,\pi^*}(m)^2: = 
\left(\log\left(\frac{1}{2\pi^*}\right) + \frac{1}{m}\log\left(\frac{2n^2}{\delta}\right)
\right)
\log\left(\frac{4n^2}{\delta}\right).
$$
\label{lm:reg}
\end{lm}

The proof of the lemma bounds the probability that $||\nabla \ell(A^*)||_\infty$ exceeds a fixed value with the sum of probabilities of two events. One of these events is a deviation event for the sum of a bounded-difference martingale sequence defined by the sequence of gradients of the log-likelihood function over a fixed horizon time; which we bound by using Azuma-Hoeffding's inequality. The other event is the probability that the sum of consensus times of $m$ independent voter model processes exceeds the value of the fixed horizon time; which we bound by using the probability tail bound for the sum of consensus times in Theorem~\ref{thm:tail}. 

Note that the bound on $||\nabla \ell(A^*)||_\infty$ in Lemma~\ref{lm:reg} involves the term $\sqrt{m/\Phi_{A^*}}$. In view of the bound on the expected consensus time in Corollary~\ref{cor:etau}, we may intuitively think of the term $\sqrt{m/\Phi_{A^*}}$ as an upper bound on the square-root of the expected number of observed time steps of the extended voter model process, i.e. $\sqrt{m\EP[\tau]}$. Note also that the bound (\ref{equ:infnorm}) in Lemma~\ref{lm:reg} remains to hold by replacing $c_{n,\delta,\pi^*}(m)$ with $c_{n,\delta,\pi^*}(1)$, in which case the right-hand side in (\ref{equ:infnorm}) scales with $m$ as $\sqrt{m}$.

We can lower bound the first-order Taylor error function as follows. 

\begin{lm} Assume that $A^*$ and $A^* + \Delta$ with $\Delta = (\Delta_1, \ldots, \Delta_n)^\top$ have a common support. Then, we have
$$
\mathcal{E}(\Delta)\geq h(\Delta;X)
$$
where
$$
h(\Delta; X) := 
\sum_{i=1}^m \sum_{t=0}^{\tau_i-1}\sum_{u=1}^n \left(\Delta_u^\top X_t^{(i)}\right)^2.
 $$
\label{lm:h0}
\end{lm}




By Lemma~\ref{lm:h0}, in order to show that the first-order Taylor error function $\mathcal{E}(\Delta)$ satisfies the RSC condition in Definition~\ref{def:rsc}, it suffices to show that the RSC condition holds for function $h(\Delta;X)$. We first show that the RSC condition holds for the expected value of $h(\Delta; X)$ for any fixed value of $\Delta$.


\begin{lm} For any $\Delta = (\Delta_1,\ldots, \Delta_n)^\top$, we have 
$$
\EP[h(\Delta; X)] \geq \kappa_1 ||\Delta ||_F^2
$$
for any $\kappa_1 > 0$ such that
$$
\kappa_1 \leq m \EP[\tau]\lambda_{\min}(\E[X_0X_0^\top]).
$$
\label{lm:h}
\end{lm}


The correlation matrix $\E[X_0 X_0^\top]$ is with respect to the stationary distribution of the extended voter process that does not include final consensus states of individual voter processes. The smallest eigenvalue $\lambda_{\min}(\E[X_0X_0^\top])$ plays an important role in Lemma~\ref{lm:h} and the results that follow. Note that by the Palm inversion formula (\ref{equ:palm}),
\begin{equation}
\lambda_{\min}(\E[X_0X_0^\top])=\frac{1}{\EP[\tau]}\lambda_{\min}\left(\EP\left[\sum_{t=0}^{\tau-1}X_tX_t^\top\right]\right).
\label{equ:corr}
\end{equation}

From (\ref{equ:corr}), it can be readily observed that
\begin{equation}
\lambda_{\mathrm{min}}(\E[X_0 X_0^\top]) \geq \frac{1}{\EP[\tau]}\lambda_{\mathrm{min}}(\EP[X_0 X_0^\top]).
\label{equ:eiglb}
\end{equation}

If the initial state distribution $\mu$ is of product-form with Bernoulli $(p)$ marginal distributions, with $0 < p < 1$, then $\lambda_{\mathrm{min}}(\EP[X_0 X_0^\top]) = p(1-p)$, and we have
$$
\lambda_{\mathrm{min}}(\E[X_0X_0^\top]) \geq p(1-p)\frac{1}{\EP[\tau]} .
$$
This bound is not tight. Tighter bounds can be obtained by analysis of the spectrum of the stationary correlation matrix $\E[X_0 X_0^\top]$ by using the Lyapunov matrix equation, which we discuss in supplementary material.
For example, for the complete graph case, when $a_{u,u} = 0$ and  $a_{u,v} = 1/(n-1)$ for all $u\neq v$, we have 
$$
\lambda_{\mathrm{min}}(\E[X_0X_0^\top])= p(1-p)\frac{n}{\EP[\tau]} (1+o(1)).
$$

We next show that for any fixed value $\Delta$, $h(\Delta;X)$ satisfies the RSC condition in Definition~\ref{def:rsc} with a prescribed probability, provided that the number of observations $m$ is sufficiently large.  

\begin{lm}
For a voter model process with parameter $A^*$ and initial distribution $\mu$, for any $\delta \in (0,1/2]$, any $S\subseteq V^2$ such that $|S|\leq s$ for some positive integer $s$, and any $\Delta\in \mathbb{C}(S,A^*)$, $h(\Delta; X)$ satisfies the RSC condition relative to $A^*$ and $S$, with curvature $\kappa = \kappa_1/2$ and tolerance $\gamma =0$, with probability at least $1-\delta$, under condition
$$
m \geq \frac{s^2}{\Phi_{A^*}}\frac{1}{\EP[\tau]^2 \lambda_{\mathrm{min}}(\E[X_0 X_0^\top])^2}c_{\delta,\pi^*}(m),
$$
where
$$
c_{\delta, \pi^*}(m) = 8\left(\log\left(\frac{1}{2\pi^*}\right)+\frac{1}{m}\log\left(\frac{2}{\delta}\right)\right)\log\left(\frac{2}{\delta}\right).
$$
\label{lm:hrsc}
\end{lm}

We next show that $h(\Delta;X)$ satisfies the RSC condition in Definition~\ref{def:rsc} for every $\Delta\in C^*$ with certain probability.  

\begin{lm}   
For a voter model process with parameter $A^*$ and initial distribution $\mu$, function $h(\Delta;X)$ satisfies the RSC condition relative to $A^*$ and $S^*$, for every $\Delta \in \mathbb{C}^*$ with probability at least $1-4/n$,
$$
h(\Delta; X)\geq \kappa' ||\Delta||_F^2 - { \gamma' }^2 \hbox{ for all } \Delta \in \mathbb{C}^*
$$
where $\kappa' = \kappa_1/8$ and ${\gamma'} = \sqrt{\kappa_1/8}||A^*_{S^*}||_{1,1}/\sqrt{s}$, provided that 
\begin{equation}
m \geq m_1 := c_1 s^2\frac{\log(1/(2\pi^*))(a+1)\log(n) +  (a+1)^2\log(n)^2}{\Phi_{A^*}\EP[\tau]^2\lambda_{\min}(\E[X_0 X_0^\top])^2}
\label{equ:m1}
\end{equation}
and
\begin{equation}
m\geq m_2 := c_2 n^3(1/\pi^*)\frac{1}{(\Phi_{A^*}\EP[\tau])^2\lambda_{\min}(\E[X_0 X_0^\top])^2}
\label{equ:m2}
\end{equation}
where
$$
a = sn\frac{\log(1/(2\pi^*))+\log(n)}{(\Phi_{A^*}\EP[\tau])\lambda_{\min}(\E[X_0 X_0^\top])}
$$
for some constants $c_1, c_2 > 0$.
\label{lm:union}
\end{lm}

The proof of Lemma~\ref{lm:union} relies on some set covering arguments to bound the probability of events indexed with $\Delta$, which takes values in the infinite set $\mathbb{C}^*$. These covering arguments require stronger conditions on the number of voter model realizations $m$ than in Lemma~\ref{lm:hrsc}, which shows that the RSC condition holds in probability, for any fixed value $\Delta$.



Consider the case when the voter model process has the stationary distribution $\pi$ of $A^*$ such that $\pi^*$ is lower bounded by a polynomial in $1/n$. Then, a sufficient condition for (\ref{equ:m1}) is that for some constant $c> 0$,
$$
m\geq c\frac{1}{\EP[\tau]}\frac{s^4 n^2}{\lambda_{\min}(\E[X_0X_0^\top])^4}\frac{\log(n)^2}{(\Phi_{A^*}\EP[\tau])^3}.
$$

We next present our main theorem that provides a bound on the parameter estimation error. 


\begin{thm} Consider the voter model process with parameter $A^*$ with support size $s$. Assume that $\hat{A}$ is a minimizer of the loss function $\mathcal{L}(A;X)$ defined by (\ref{equ:loss}) with the regularization parameter 
$$
\lambda_m = 2\sqrt{2}\frac{c_{n,\pi^*}}{\alpha\sqrt{\Phi_{A^*}}}\sqrt{m},
$$
and conditions (\ref{equ:m1}) and (\ref{equ:m2}) hold. Then, for some constant $c > 0$, with probability at least $1-5/n$, 
\begin{equation}
||\hat{A}-A^*||_F^2  \leq   
c\frac{s c_{n,\pi^*}^2}{\alpha^2 (\Phi_{A^*}\E^0[\tau])^2\lambda_{\min}(\E[X_0X_0^\top])^2}\Phi_{A^*}\frac{1}{m} 
\label{equ:frob}
\end{equation}
where
$$
c_{n,\pi^*}^2 = 
\left(\log\left(\frac{1}{2\pi^*}\right) + \log\left(2n^3\right)
\right)
\log\left(4n^3\right).
$$
%
\label{thm:paramerr}
\end{thm}



Theorem~\ref{thm:paramerr} gives us a bound on the parameter estimation error in (\ref{equ:frob}) that holds with high probability under sufficient conditions (\ref{equ:m1}) and (\ref{equ:m2}) for consistency of the estimator. 
For any initial distribution $\mu$ and parameter $A^*$ such that $\EP[\tau]$ is equal to $1/\Phi_{A^*}$ up to a poly-logarithmic factor in $n$, the term $\Phi_{A^*}\EP[\tau]$ contributes only poly-logarithmic factors in (\ref{equ:frob}). In the case when the product $\Phi_{A^*}\EP[\tau]$ is poly-logarithmic in $n$, for asymptotically large $m$,
$$
m ||\hat{A}-A^*||_F^2 = \tilde{O}\left(\frac{s\Phi_{A^*}}{\alpha^2 \lambda_{\min}(\E[X_0 X_0^\top])^2}\right).
$$

\subsection{Sampling complexity lower bound} In this section, we show a lower bound on the sampling complexity for the parameter estimation of the voter model. This lower bound is derived using the framework of locally stable estimators \cite{JP19}. Intuitively, a locally stable estimator of a parameter is robust to small perturbations of the true parameter value. 

To make a precise definition, let $\mathbb{B}(A^*,r)$ be the ball with centre point $A^*$ and radius $r$, i.e. $\mathbb{B}(A^*,r) = \{A\in \Theta: ||A^*-A||_F\leq r\}$, where $\Theta$ is some set of parameter values. Let $\P_{A}[\cdot]$ denote the probability distribution under a statistical model with parameter $A$. The notion of locally stable estimators is defined as follows.

\begin{df} An estimator $\hat{A}$ is said to be $(\epsilon,\delta)$-locally stable in $A^*$ with parameters $\epsilon>0$ and $\delta \in (0,1)$, if there exists a finite $m_0$ such that for all $m\geq m_0$ and $A\in \mathbb{B}(A^*,3\epsilon)$,
$$
\P_{A}[||\hat{A}-A||_F\leq \epsilon] \geq 1-\delta.
$$
\end{df}

Roughly speaking, for any locally stable estimator in $A^*$, a given bound on the parameter estimation error holds in probability with respect to any parameter $A$ that is in a neighborhood of $A^*$. In our setting, we let $\Theta$ be the set of $n\times n$ substochastic matrices. 

The following theorem gives a lower bound on the sampling complexity for the class of locally stable estimators.


\begin{thm} Assume that $A^*$ is a stochastic matrix such that each element in its support has value at least $\alpha > 0$. Let $q_1$ be the eigenvector corresponding to the smallest eigenvalue of the correlation matrix $\E[X_0X_0^\top]$ of the extended voter process with parameter $A^*$ and initial state distribution $\mu$. 

Then, for any $(\epsilon,\delta)$-locally stable estimator in $A^*$, such that $\delta\in (0,1)$ and $\epsilon \in (0,\min\{1/(2|q_1^\top \vec{1}|),\alpha/4\})$, it holds
\begin{equation}
m \EP[\tau] \lambda_{\min}(\E[X_0 X_0^\top]) \geq \frac{\alpha}{16}\frac{1}{\epsilon^2}\log\left(\frac{1}{2.4\delta}\right).
\label{equ:lbcond}
\end{equation}
Moreover, (\ref{equ:lbcond}) holds under stronger condition $\epsilon \in (0, \min\{1/(2\sqrt{n}), \alpha/4\})$.
\label{thm:lb}
\end{thm}

The proof of the theorem follows similar arguments as in the proof of a sampling complexity lower bound for linear discrete-time dynamical systems with additive Gaussian noise \cite{JP19}. The extended voter process requires us to study a different discrete-time random dynamical system, and addressing certain technical points that arise due to Bernoulli random variables and constraints on the node interaction parameters.

The result in Theorem~\ref{thm:lb} shows us that the upper bound in Theorem~\ref{thm:paramerr} is tight with respect to the relation between $||\hat{A}-A^*||_F$ and $m$ for small estimation error case. If $\delta$ is polynomial in $1/n$ and $\epsilon^2 \leq \min\{1/n, \alpha^2\}/16$, then from (\ref{equ:lbcond}) we have
$$
m\epsilon^2 = \Omega\left(\frac{\alpha}{\EP[\tau]\lambda_{\min}(\E[X_0 X_0^\top])}\log(n)\right).
$$

\newpage

\appendix

\section{Mathematical background}

\subsection{KL divergence bounds} 
\label{sec:kl}

Let $p$ and $q$ be two distributions on $\mathcal{X}$ such that $p(x) = 0$ for all $x\in \mathcal{X}$ such that $q(x) = 0$. The KL divergence between $p$ and $q$ is defined by
$$
\mathrm{KL}(p\mid \mid q) = \sum_{x\in \mathcal{X}} p(x) \log\left(\frac{p(x)}{q(x)}\right).
$$

The total variation distance between $p$ and $q$ is defined by
$$
\delta(p,q) = \frac{1}{2}\sum_{x\in {\mathcal X}} |p(x)-q(x)|.
$$

The Pinsker's inequality is
\begin{equation}
\mathrm{KL}(p\mid\mid q) \geq 2\delta(p,q)^2.
\label{equ:pinsker}
\end{equation}

Because $(\sum_{x\in \mathcal{X}} |p(x)-q(x)|)^2 \geq ||p-q||^2$, it follows
\begin{equation}
\mathrm{KL}(p\mid\mid q)\geq \frac{1}{2}||p-q||^2.
\label{equ:kllb}
\end{equation}

Let $\alpha > 0$ be a constant such that $q(x) \geq \alpha$ for all $x\in \mathcal{X}$ such that $q(x) > 0$. Then, we have the following upper bound for the KL divergence:
\begin{equation}
\mathrm{KL}(p\mid\mid q) \leq \frac{1}{\alpha} ||p-q||^2.
\label{equ:klub}
\end{equation}

The proof is easy and is provided here for completeness:
\begin{eqnarray*}
\mathrm{KL}(p\mid\mid q) &=& \sum_{x\in \mathcal{X}: q(x) > 0} p(x)\log\left(\frac{p(x)}{q(x)}\right)\\
&=& \sum_{x\in \mathcal{X}: q(x) > 0} q(x)\left(\frac{p(x)-q(x)}{q(x)}+1\right)\log\left(\frac{p(x)-q(x)}{q(x)}+1\right)\\
&\leq & \sum_{x\in \mathcal{X}: q(x) > 0} q(x) \left(\frac{p(x)-q(x)}{q(x)}+1\right) \frac{p(x)-q(x)}{q(x)} \\
&=& \sum_{x\in \mathcal{X}: q(x) > 0} \frac{(p(x)-q(x))^2}{q(x)}\\
&\leq & \frac{1}{\alpha}||p-q||^2
\end{eqnarray*}
where the first inequality follows by the fact that $\log(x+1)\leq x$ for all $x > -1$ and the last inequality follows by the definition of $\alpha$.

Suppose that $p$ and $q$ are two Bernoulli distributions with parameters $\bar{p}$ and $\bar{q}$, respectively. With a slight abuse of notation, let $\mathrm{KL}(\bar{p}\mid\mid \bar{q})$ denote the KL divergence between $p$ and $q$. Using Pinsker's inequality (\ref{equ:pinsker}) and $\delta(p,q)=\frac{1}{2}(|\bar{p}-\bar{q}|+|\bar{p}-\bar{q}|) = |\bar{p}-\bar{q}|$, we have the lower bound
\begin{equation}
\mathrm{KL}(\bar{p}\mid\mid \bar{q}) \geq 2 (\bar{p}-\bar{q})^2.
\label{equ:BerKLlb}
\end{equation}

Using (\ref{equ:klub}), under $\alpha\leq \bar{q} \leq 1-\alpha$, we have the upper bound
\begin{equation}
\mathrm{KL}(\bar{p}\mid\mid \bar{q}) \leq \frac{2}{\alpha}(\bar{p}-\bar{q})^2. 
\label{equ:BerKLub}
\end{equation}

\subsection{Concentration of measure inequalities}

\begin{thm}[Azuma-Hoeffding]  Assume $X_0, X_1, \ldots$ is a martingale sequence such that $|X_i-X_{i-1}|\leq c_i$ almost surely. Then, for all positive integers $N$ and all $\epsilon > 0$,
$$
\P[X_N - X_0 \geq \epsilon] \leq \exp\left(-\frac{\epsilon^2}{2\sum_{i=1}^N c_i^2}\right)
$$
with an identical bound for the other tail.
\label{thm:azuma}
\end{thm}

\subsection{Covering and metric entropy}
\label{sec:covering}

Let $(M,\rho)$ be a metric space, where $M$ is a set and $\rho: M\times M \rightarrow \reals_+$ is a metric. 

An \emph{$\epsilon$-covering} of $M$ in metric $\rho$ is a collection of points $\{x_1, \ldots, x_k\}\subset M$ such that for every $x\in M$, $\rho(x,x_i)\leq \epsilon$, for some $i\in \{1,\ldots, k\}$. 

The \emph{$\epsilon$-covering number}, denoted as $N(\epsilon,M,\rho)$, is the cardinality of the smallest $\epsilon$-covering of $M$ in metric $\rho$. In other words, $N(\epsilon,M,\rho)$ is the minimum number of balls with radius $\epsilon$ under metric $\rho$ required to cover $M$. 

Let $B_M(x,\epsilon,\rho)$ be the closed ball with center $x$ and radius $\epsilon$ under metric $\rho$, i.e. 
$$
B_M(x,\epsilon,\rho) = \{y\in M: \rho(x,y)\leq \epsilon\}.
$$
Then, 
$$
N(\epsilon, M,\rho)= \min\{k\in \mathbb{N}: \exists x_1, \ldots, x_k: M\subset \cup_{i=1}^k B_M(x_i,\epsilon,\rho)\}.
$$
The \emph{metric entropy} is defined as the logarithm of the covering number, i.e. $\log(N(\epsilon, M, \rho))$. 

The \emph{dyadic entropy number} $\epsilon_k(M,\rho)$ is defined as
$$
\epsilon_k(M,\rho) = \inf\{\epsilon > 0: N(\epsilon,M,\rho) \leq 2^{k-1}\}.
$$
Note that $\epsilon_k(M,\rho)\leq \epsilon$ if and only if $\log(N(\epsilon,M,\rho)) \leq k$.

Let $M = \reals^d$ and let $B_q(r)$ be the closed ball with center $0$ and radius $r$ under metric $\ell_q$. By \cite{https://doi.org/10.48550/arxiv.0910.2042}, for every $q\in (0,1]$ and $p \in [1,\infty]$ such that $p > q$, there exists a constant $c_{q,p}$ such that
$$
\log(N(\epsilon,B_q(r), \ell_p)) \leq c_{q,p} r^{p/(p-q)} \left(\frac{1}{\epsilon}\right)^{1/(1/q-1/p)} \log(d), \hbox{ for all } \epsilon \in (0,r^{1/q}).
$$

In particular, for $q = 1$ and $p = 2$, for some constant $c > 0$,
\begin{equation}
\log(N(\epsilon,B_1(r),\ell_2)) \leq c \left(\frac{r}{\epsilon}\right)^2 \log(d), \hbox{ for all } \epsilon \in (0,r).
\label{equ:logn}
\end{equation}



In the following lemma we provide a bound on the metric entropy for a certain metric $\rho$ that is of interest for our parameter estimation problem. A similar lemma was stated in \cite{SparseBer2019} (Lemma~A.8) without a proof. Our lemma shows that parameters $c_1$ and $c_2$ in the lemma are not constants but depend on matrix $X$. This has significant implications on required conditions when applying the lemma to the parameter estimation problem.
\begin{lm}
Let $X$ be a real $T\times n$ matrix with each column having $\ell_2$ norm bounded by $\sqrt{T}$. Let $B_1(r) = \{\Delta \in \reals^{n\times n}: ||\Delta||_1 \leq r\}$ and $\rho(\Delta,\Delta')=1/\sqrt{T}||X(\Delta-\Delta')^\top ||_F$. Then, there exist constant $c > 0$ such that the metric entropy of $B_1(r)$ in $\rho$ is bounded as
$$
\log(N(\epsilon, B_1(r), \rho))\leq c_1 \left(\frac{r}{\epsilon}\right)^2 \log(n), \hbox{ for all } \epsilon \in (0,c_2 r]
$$
where $c_1 = c\sigma_{\max}(X)^2/T$ and $c_2 = \sigma_{\max}(X)/\sqrt{T}$.
\label{lm:n}
\end{lm}

\begin{proof} For any real $k\times n$ matrix $M$ and real $n\times n$ matrix $D$ with $k\geq n$, we have
$$
||M D||_F \leq ||M|| ||D||_F = \sigma_{\max}(M)||D||_F
$$
where $\sigma_{\max}(M)$ denotes the largest singular value of $M$.

Hence, we have
$$
1/\sqrt{T}||X(\Delta-\Delta')^\top ||_F \leq (\sigma_{\max}(X)/\sqrt{T}) ||\Delta-\Delta'||_F.
$$
Let $\rho'(\Delta,\Delta') = (\sigma_{\max}(X)/\sqrt{T}) ||\Delta-\Delta'||_F$, and $N(\epsilon, B_1(r), \rho')$ be the minimum number of balls with radius $\epsilon$ under metric $\rho'$ to cover $B_1(r)$, and $\epsilon' = \sqrt{T}\epsilon/\sigma_{\max}(X)$. Then, we have
\begin{eqnarray*}
N(\epsilon, B_1(r), \rho) &\leq& N(\epsilon, B_1(r), \rho') \\
&=& N(\epsilon', B_1(r), \ell_2) \\
&\leq& e^{c_1\left(\frac{r}{\epsilon}\right)^2\log(n)}
\end{eqnarray*}
for all $\epsilon \in (0,c_2r)$. The first inequality comes from the fact that $\rho'$ balls take up less space than $\rho$ balls, so it would
take more of them to do the covering, and the second inequality comes from (\ref{equ:logn}).
\end{proof}



\section{Consensus time}

\subsection{Proof of Lemma~\ref{lm:drift}}

Lemma~\ref{lm:drift} is a known result. We provide a proof for the sake of completeness. 

We claim that for every $t\geq 0$ and $x\notin C$,
\begin{equation}
V(x)\geq \lambda^{-t}\E_x[V(X_t)\ind_{\{\tau_C > t\}}] + \sum_{s=1}^t \lambda^{-s}\P_x[\tau_C = s].
\label{equ:drift1}
\end{equation}

We will prove this by induction on $t$. Base case $t=0$ trivially holds. For the induction step, assume that (\ref{equ:drift1}) holds for $t$, and we are going to prove that it then holds for $t+1$. For any $x\notin C$,
\begin{eqnarray*}
&& \E_x[V(X_t)\ind_{\{\tau_C > t\}}]\\
 &\geq & \lambda^{-1}\E_x[PV(X_{t+1})\ind_{\{\tau_C > t\}}]\\
&=& \lambda^{-1}\E_x[V(X_{t+1})\mid \ind_{\{\tau_C > t\}}]\\
&=& \lambda^{-1}\left(\E_x[V(X_{t+1})\ind_{\{\tau_C > t+1\}}] + \E_x[V(X_{t+1})\ind_{\{\tau_C = t+1\}}]\right)\\
& \geq & \lambda^{-1}\E_x[V(X_{t+1})\ind_{\{\tau_C > t+1\}}] + \lambda^{-1}\P_x[\tau_C=t+1]
\end{eqnarray*}
where the first inequality is by the drift condition and the last inequality is because $V(x)\geq 1$ for all $x\in \mathcal{X}$.

Now, use the last inequality to obtain that (\ref{equ:drift1}) holds for $t+1$. It follows that
$$
V(x) \geq \lambda^{-t}\P_x[\tau_C > t] + \sum_{s=1}^t \lambda^{-s}\P_x[\tau_C = s].
$$
By letting $t$ goes to infinity, we have
$$
V(x) \geq \sum_{s=1}^\infty \lambda^{-s}\P[\tau_C = s] = \E_x[\lambda^{-\tau_C}].
$$

\subsection{Proof of Lemma ~\ref{lem:drift}}
\label{app:drift}

We first consider the case when $x\in \{0,1\}^n$ such that $x\notin C$. Note that
\begin{eqnarray}
&& \E[V_\pi(X_{t+1}) - V_\pi(X_t) \mid X_t = x] \nonumber\\
& = & \pi^\top \E[X_{t+1} \mid X_t = x] \nonumber \\
&& - \pi^\top x - (\E[(\pi^\top X_{t+1})^2\mid X_t = x] - (\pi^\top x)^2)\nonumber\\
&& -(\E[(\pi^\top X_{t+1})^2\mid X_t = x] - (\pi^\top x)^2) \label{equ:drift}
\end{eqnarray}
where the last equation is by the fact that $\pi^\top X_t$ is a martingale. 

Let 
$$
D_A(x) = \mathrm{diag}(a_1^\top x (1-a_1^\top x), \ldots, a_n^\top x (1-a_n^\top x)).
$$

Note that
\begin{eqnarray*}
&& \E[(\pi^\top X_{t+1})^2\mid X_t = x] \\
&=& \pi^\top \E[X_{t+1}X_{t+1}^\top\mid X_t = x]\pi\\
&=& \pi^\top (A\E[X_t X_t^\top \mid X_t = x]A^\top + \E[D_A(X_t)\mid X_t = x])\pi\\
&=& \pi^\top x x^\top\pi  + \pi^\top  D_A(x) \pi\\
&=& (\pi^\top x)^2 + \pi^\top  D_A(x) \pi.
\end{eqnarray*}

Plugging the derived identity in (\ref{equ:drift}), we obtain 
$$
\E[V_\pi(X_{t+1}) - V_\pi(X_t) \mid X_t = x] = - \sum_{u=1}^n \pi_u^2 V_{a_u}(x).
$$

Now, consider the case when $x\in C$. Then, $V_\pi(x) = 0$, and we have
$$
\E[V_\pi(X_{t+1}) - V_\pi(X_t) \mid X_t = x] = \EP[V_\pi(X_0)].
$$
Hence, we have shown that
\begin{eqnarray*}
&& \E[V_\pi(X_{t+1}) - V_\pi(X_t) \mid X_t = x] \\
& = & - \sum_{u=1}^n \pi_u^2 V_{a_u}(x) + \EP[V_\pi(X_0)]\ind_{\{x\in C\}}.
\end{eqnarray*}

\subsection{Proof of Lemma~\ref{lm:drift-async}}

We first consider the case when $x\in \{0,1\}^n$ such that $x\notin C$. In this case, we have
\begin{eqnarray*}
&& \E[V_\pi(X_{t+1})-V_\pi(X_t) \mid X_t = x] \\
&=& \pi^\top \E[X_{t+1}\mid X_t = x] - \pi^\top \E[X_{t+1}X_{t+1}^\top\mid X_t = x]\pi - \pi^\top x + \pi^\top x x^\top \pi.
\end{eqnarray*}

Now, note
$$
\E[X_{t+1}\mid X_t = x] = \frac{1}{n}Ax + \left(1-\frac{1}{n}\right)x.
$$
Hence, $\pi^\top \E[X_{t+1}\mid X_t = x] = \pi^\top x$. It follows
$$
\E[V_\pi(X_{t+1})-V_\pi(X_t) \mid X_t = x]
= - \pi^\top \E[X_{t+1}X_{t+1}^\top\mid X_t = x]\pi + \pi^\top x x^\top \pi.
$$

For $u\neq v$, we have
$$
\E[X_{t+1,u}X_{t+1,v}\mid X_t = x] = \frac{1}{n}a_u^\top x x_v + \frac{1}{n} x_u a_v^\top x + \left(1-\frac{2}{n}\right)x_u x_v
$$
and
$$
\E[X_{t+1,u}X_{t+1,u} \mid X_t = x] = \frac{1}{n}a_u^\top x + \left(1-\frac{1}{n}\right) x_u.
$$
In a matrix notation, we have
$$
\E[X_{t+1}X_{t+1}^\top \mid X_t = x] = \frac{1}{n} (A x) x^\top + \frac{1}{n} x (Ax)^\top + \left(1-\frac{2}{n}\right) xx^\top + \frac{1}{n} D_A(x)
$$
where $D_A(x)$ is the diagonal matrix with diagonal elements
$$
(D_A(x))_{u,u} = x_u(1-a_u^\top x) + (1-x_u) a_u^\top x.
$$
It follows that
$$
\pi^\top \E[X_{t+1}X_{t+1}^\top \mid X_t = x] \pi = \pi^\top x x^\top \pi + \frac{1}{n} \sum_{u=1}^n \pi_u^2 (x_u(1-a_u^\top x) + (1-x_u) a_u^\top x)
$$
and, hence,
$$
\E[V_\pi(X_{t+1})-V_\pi(X_t) \mid X_t = x] = - \frac{1}{n} \sum_{u=1}^n \pi_u^2 (x_u(1-a_u^\top x) + (1-x_u) a_u^\top x).
$$

For the case when $x\in C$, we have
$$
\E[V_\pi(X_{t+1})-V_\pi(X_t) \mid X_t = x] = \EP[V_\pi(X_0)].
$$

\subsection{Proof of Lemma~\ref{lem:phi-cond-bias}}

It is can be readily checked that $\pi_u = d_u/d(V)$, for $u\in V$.

Next, note
\begin{eqnarray*}
\sum_{u\in V} \pi_u^2 a_u^\top x (1-a_u^\top x) &=& \frac{1}{4d(V)^2} \sum_{u\in V} (d_u \ind_{\{u\in S\}} + d_u(S))(d_u \ind_{\{u\in S^c\}}+d_u(S^c))\\
&\geq & \frac{d_{\min}}{4d(V)^2}\left(\sum_{u\in S} d_u(S^c) + \sum_{u\in S^c }d_u(S)\right)\\
&=& \frac{d_{\min}}{2 d(V)^2}|E(S,S^c)|.
\end{eqnarray*}
It follows
\begin{eqnarray*}
\frac{\sum_{u\in V}\pi_u^2 a_u^\top x(1-a_u^\top x)}{\pi^\top x(1-\pi^\top x)} & \geq & \frac{d_{\min}}{2}\frac{|E(S,S^c)|}{d(S)d(S^c)}\\
&\geq & \frac{d_{\min}}{2d(V)} \frac{|E(S,S^c)|}{\min\{d(S),d(S^c)\}}.
\end{eqnarray*}

Hence, we have
$$
\frac{1}{\Phi_A}\leq 2\frac{d(V)}{d_{\min}}\frac{1}{\Phi(G)}.
$$

\subsection{Proof of Lemma~\ref{lem:psi-cond}}


Note that
\begin{eqnarray*} 
\E[|\pi^\top (I-Z)x|] &\leq &\sum_{u=1}^n \pi_u\E\left[\left|x_u -\sum_{v=1}^n Z_{u,v}x_v\right|\right]\\
&=& \sum_{u=1}^n \pi_u [x_u(1-a_u^\top x) + (1-x_u)(a_u^\top x)]\\
&=& 2\sum_{u=1}^n \sum_{v=1}^n \pi_u a_{u,v}x_u(1-x_v)
\end{eqnarray*}
where the inequality is by Jensen's inequality. For $A$ according to (\ref{equ:auv}), we have
$$
\sum_{u=1}^n \sum_{v=1}^n \pi_u a_{u,v}x_u(1-x_v) = \frac{1}{2d(V)}|E(S,S^c)|
$$
and 
$$
\min\{\pi^\top x, 1-\pi^\top x\} = \frac{1}{d(V)} \min\{d(S),d(S^c)\}.
$$
Hence, it follows
$$
\tilde{\Psi}_A \leq \Phi(G).
$$

\subsection{Proof of Lemma~\ref{lem:phi-cond}}

For every $S\subseteq V$, we have
\begin{eqnarray*}
|E_2(S,S^c)| &=& \sum_{u\in V} d_u(S)d_u(S^c) \\
&=& \sum_{u\in S^c} d_u(S)d_u(S^c) + \sum_{u\in S} d_u(S)d_u(S^c) \\
&\geq& \sum_{u\in S^c} d_u(S) + \sum_{u\in S}d_u(S^c)\\
&=& 2|E(S,S^c)|
\end{eqnarray*}
where the inequality holds by the fact that $d_u(S) \geq 1$ when $u \in S$, , for any set $S\subset V$, because each vertex in $G$ has a self-loop.

Thus, we have
$$
\sum_{u\in V}\pi_u^2 a_u^\top x (1-a_u^\top x) = \frac{1}{d(V)^2} |E_2(S,S^c)| \geq \frac{2}{d(V)^2}|E(S,S^c)|.
$$
Combining with
$$
\pi^\top x (1-\pi^\top x) =  \frac{d(S)d(S^c)}{d(V)^2} \leq \frac{1}{d(V)}\min\{d(S),d(S^c)\},
$$
and (\ref{equ:psistar}) and (\ref{equ:phigraph}), we have 
$$
\frac{1}{\Phi_A} \leq \frac{1}{2}d(V)\frac{1}{\Phi(G)}.
$$

\subsection{Further discussion of our expected consensus time bound}
\label{sec:fd}

We consider the case where in each time step, each node samples a node uniformly at random from its neighborhood set, i.e. we consider $A$ defined as
\begin{equation}
a_{u,v} = \frac{1}{d_u}\ind_{\{(u,v) \in E\}}, \hbox{ for } u,v\in V.
\label{equ:auv}
\end{equation}
It is readily checked that $\pi_v = d_v/d(V)$, for $v\in V$, and 
\begin{equation}
\Phi_A = \min_{S\subset V: 0<|S|<n} \frac{|E_2(S,S^c)|}{d(S)d(S^c)}
\label{equ:phigraph}
\end{equation}
where $E_2(S,S^c)$ is the set of paths consisting of two edges connecting  $S$ and $S^c$. We will assume that every node in $G$ has a self-loop. The node interaction matrix $A$ in this case can be interpreted to correspond to a lazy random walk that has a higher probability of remaining at a vertex for smaller degree vertices. In other words, a higher-degree vertex has a higher probability of adopting the state of a neighbor. 


\begin{lm} Assume that node interaction matrix $A$ is according to (\ref{equ:auv}) and every node in $G$ has a self-loop. Then, we have
$$
\frac{1}{\Phi_A} \leq \frac{1}{2}d(V)\frac{1}{\Phi(G)}.
$$
\label{lem:phi-cond}
\end{lm}

Together with Corollary~\ref{cor:etau}, Lemma~\ref{lem:phi-cond} implies $\EP[\tau] = O((d(V)/\Phi(G))\log(n))$. By the same arguments as in the proof of Lemma~\ref{lem:psi-cond}, we have the following lemma.

\begin{lm} 
Assume that node interaction matrix $A$ is according to (\ref{equ:auv}). Then, we have
$$
\tilde{\Psi}_A \leq 2 \Phi(G).
$$
\label{lem:psi-cond-2}
\end{lm}

From Lemmas \ref{lem:phi-cond} and \ref{lem:psi-cond-2}, we have the following lemma.

\begin{lm} Assume that node interaction matrix $A$ is according to (\ref{equ:auv}) and every node in $G$ has a self-loop. Then, we have
$$
\frac{1}{\Phi_A} \leq d_{\min}\frac{1}{\Psi_A}.
$$
\label{lem:phi-psi}
\end{lm}

Together with Corollary~\ref{cor:etau}, Lemma~\ref{lem:phi-psi} implies that our bound on the expected consensus time is at most a $O(d_{\min}\log(n))$ factor of the expected consensus time bound in \cite{CR16}.

For the asynchronous discrete-time voter model with $A$ according to (\ref{equ:auv}), we have 
$$
\EP[\tau]\leq \frac{1}{\Phi_A'}\log\left(\frac{d(V)}{2 d_{\min}}\right)
$$ 
where   
$$
\Phi_A' = \frac{1}{n}\min_{S\subset V: 0 < |S| < n}\left\{\frac{\sum_{u\in S, v\in S^c} (d_u+d_v)\ind_{\{(u,v)\in E\}}}{d(S)d(S^c)}\right\}.
$$

From \cite{CR16}, $\EP[\tau]\leq 64/\Psi_A$ where
$$
\Psi_A = \frac{2}{n}\frac{d_{\min}}{d(V)}\min_{S\subset V: 0 < |S| < n} \frac{|E(S,S^c)|}{\min\{d(S),d(S^c)\}}.
$$

From the above relations, we have the following lemma.
\begin{lm} For the asynchronous discrete-time voter model with $A$ according to (\ref{equ:auv}), we have 
$$
\frac{1}{\Phi_A'} \leq \frac{1}{\Psi_A}.
$$
\end{lm}

The last lemma implies that our bound on the expected consensus time is within a logarithmic factor in $n$ to the expected consensus time bound in \cite{CR16}.

\section{Parameter estimation}

\subsection{Proof of Lemma~\ref{lm:reg}}
\label{app:reg}

By the union bound, for any $\Lambda\geq 0$,
$$
\PP[||\nabla \ell(A^*;X)||_\infty \geq \Lambda] \leq n^2 \max_{(u,v)\in V^2}\PP\left[\left|\frac{\partial}{\partial a_{u,v}} \ell(A^*;X)\right|\geq \Lambda\right].
$$

Fix $u,v \in V$. Let us define
$$
Y_{t}^{(i)} = \left(\frac{X_{t,u}^{(i)}}{{a_u^*}^\top X_{t-1}^{(i)}} - \frac{1-X_{t,u}^{(i)}}{1-{a_u^*}^\top X_{t-1}^{(i)}}\right)X_{t-1,v}^{(i)} \ind_{\{0<{a_u^*}^\top X_{t-1}^{(i)} < 1\}}.
$$

Note that
\begin{equation}
\PP\left[\left|\frac{\partial}{\partial a_{u,v}} \ell(A^*;X)\right|\geq \Lambda \right] = 
\PP\left[\left| \sum_{i=1}^m \sum_{t=1}^{\tau_i} Y_{t}^{(i)} \right|\geq \Lambda \right].
\label{equ:plpz}
\end{equation}

For $0 < s\leq \sum_{i=1}^m \tau_i$, let us define
$$
Y_{s} = Y_{{s-\sum_{i=1}^{k-1}\tau_i}}^{(k)}, \hbox{ for } \sum_{i=1}^{k-1} \tau_i < s \leq \sum_{i=1}^{k} \tau_i \hbox{ and } k\in [m] 
$$
where $\sum_{i=1}^0 \tau_i \equiv 0$, and 
$$
Y_{s} = Y_{\tau_m}^{(m)} = 0, \hbox{ for } s > \sum_{i=1}^{m}\tau_i.
$$


From (\ref{equ:plpz}), for any $T>0$,
$$
\PP\left[\left|\frac{\partial}{\partial a_{u,v}} \ell(A^*;X)\right|\geq \Lambda \right] 
\leq \PP\left[\left | \sum_{t=1}^T Y_{t}\right | \geq \Lambda \right] + \PP\left[\sum_{i=1}^m \tau_i> T\right].
$$

For every $t\geq 0$, we have
$$
\E[Y_{t} \mid \mathcal{F}_{t-1}] = 0.
$$
Hence, $Y_{1}, Y_{2}, \ldots$ is a martingale difference sequence. For all $t\geq 1$, $|Y_{t}|\leq 1/\alpha$, with probability $1$. Hence, we can apply Azuma-Hoeffding's inequality (Theorem~\ref{thm:azuma}) to obtain
\begin{equation}
 \P\left[\left | \sum_{t=1}^T Y_{t}\right | \geq \Lambda \right] \leq 2e^{-\frac{\alpha^2 \Lambda^2}{2T}}.
\label{equ:t1}
\end{equation}

By Theorem~\ref{thm:tail}, we have 
\begin{equation}
\PP\left[\sum_{i=1}^m \tau_i > T\right] \leq e^{-cm}
\label{equ:t2}
\end{equation}
for any 
$$
T \geq m \left(\log\left(\frac{1}{2\pi^*}\right)+c\right)\frac{1}{\Phi_{A^*}}.
$$

Let $\delta \in (0,1]$. We require that the right-hand sides of the inequalities in (\ref{equ:t1}) and (\ref{equ:t2}) are less than or equal to $\delta/(2n^2)$. This yields
$$
\Lambda \geq \sqrt{2} \sqrt{T} \frac{1}{\alpha}\sqrt{\log\left(\frac{4n^2}{\delta}\right)}
$$
and
$$
T \geq  m 
\left(\log\left(\frac{1}{2\pi^*}\right) + \frac{1}{m}\log\left(\frac{2n^2}{\delta}\right)
\right)\frac{1}{\Phi_{A^*}}.
$$

Hence, with probability at least $1-\delta$, 
$$
||\nabla \ell(A^*)||_\infty \leq \frac{\sqrt{2}}{\alpha} 
\sqrt{m}
\frac{1}{\sqrt{\Phi_{A^*}}}
\sqrt{
\left(\log\left(\frac{1}{2\pi^*}\right) + \frac{1}{m}\log\left(\frac{2n^2}{\delta}\right)
\right)
\log\left(\frac{4n^2}{\delta}\right)
}.
$$

\subsection{Proof of Lemma~\ref{lm:h0}}

By a limited Taylor expansion, for some $\lambda \in [0,1]$,
$$
\mathcal{E}(\Delta) = vec(\Delta)^\top \nabla^2 (-\ell(A^*+\lambda \Delta)) vec(\Delta).
$$



By the properties of the Hessian matrix of the negative log-likelihood function, (\ref{equ:hessianloglik1}) and (\ref{equ:hessianloglik2}), we have 
$$
\frac{\partial^2}{\partial a_{u,v}\partial a_{u,w}}(-\ell(A; X)) \geq 
\sum_{i=1}^m \sum_{t=0}^{\tau_i-1} X_{t,v}^{(i)}X_{t,w}^{(i)}\ind_{\{0 < a_u^\top X_t^{(i)} < 1\}}.
$$

Note that
\begin{eqnarray*}
&& vec(\Delta)^\top \nabla^2(-\ell(A;X)) vec(\Delta) \\
&\geq & 
\sum_{u=1}^n\sum_{v=1}^n \sum_{w=1}^n \Delta_{u,v} \left(\sum_{i=1}^m \sum_{t=0}^{\tau_i-1} X_{t,v}^{(i)}X_{t,w}^{(i)} \ind_{\{0 < a_u^\top X_t^{(i)} < 1\}}\right)\Delta_{u,w}\\
&=& 
\sum_{i=1}^{m}\sum_{t=0}^{\tau_i-1}\sum_{u=1}^n\left(\Delta_u^\top X_t^{(i)}\right)^2\ind_{\{0 < a_u^\top X_t^{(i)} < 1\}}.
\end{eqnarray*}

Hence,
$$
vec(\Delta)^\top \nabla^2(-\ell(A;X)) vec(\Delta) \geq h(A,\Delta; X)
$$
where
$$
h(A,\Delta; X) = 
\sum_{i=1}^{m}\sum_{t=0}^{\tau_i-1}\sum_{u=1}^n\left(\Delta_u^\top X_t^{(i)}\right)^2\ind_{\{0 < a_u^\top X_t^{(i)} < 1\}}.
$$

For $A = A^* + \lambda \Delta = (1-\lambda) A^* + \lambda (A^* + \Delta)$, since $A^*$ and $A^* + \Delta$ have the same support, we have the following properties. If $a_u^\top X_t^{(i)} = 0$, then $\Delta_u^\top X_t^{(i)} = 0$ because ${a_u^*}^\top X_t^{(i)} = 0$. If $a_u^\top X_t^{(i)} = 1$, then again $\Delta_u^\top X_t^{(i)} = 0$ because ${a_u^*}^\top X_t^{(i)} = 1$. It thus follows that for any $A$ with the same support as $A^*$, we have
$$
h(A,\Delta; X) = 
h(\Delta; X)
:= \sum_{i=1}^{m}\sum_{t=0}^{\tau_i-1}\sum_{u=1}^n\left(\Delta_u^\top X_t^{(i)}\right)^2.
$$

\subsection{Proof of Lemma~\ref{lm:h}}
\label{app:h}

We consider
$$
\EP[h(\Delta;X)] = 
m \EP[\tau] \sum_{u=1}^n \E\left[\left(\Delta_u^\top X_0\right)^2\right].
$$

The following relations hold
\begin{eqnarray*}
\sum_{u=1}^n (\Delta_u^\top x)^2 &=& \sum_{u=1}^n x^\top \Delta_u \Delta_u^\top x\\
&=& x^\top \left(\sum_{u=1}^n \Delta_u \Delta_u^\top\right) x\\
&=& x^\top \Delta^\top \Delta x\\
&=& \langle \Delta x, \Delta x\rangle\\
&=& \mathrm{tr}((\Delta x)^\top \Delta x)\\
&=& \mathrm{tr}(x^\top \Delta^\top \Delta x)\\
&=& \mathrm{tr}(x x^\top \Delta^\top \Delta). 
\end{eqnarray*}

Hence, we have
\begin{eqnarray*}
\sum_{u=1}^n \E\left[\left(\Delta_u^\top X_0\right)^2\right] 
&=& \mathrm{tr}(\E[X_0X_0^\top]
\Delta^\top \Delta)\\
&\geq & \lambda_{\min}(\E[X_0X_0^\top])||\Delta||_F^2.
\end{eqnarray*}

It follows that $\EP[h(\Delta; X)] \geq \kappa_1 ||\Delta||_2^2$, for every $\kappa_1 > 0$ such that
$$
\kappa_1 \leq m \EP[\tau] \lambda_{\min}(\E[X_0X_0^\top]).
$$

\subsection{Proof of Lemma~\ref{lm:hrsc}}
\label{app:hrsc}

We first show a concentration bound for random variable $h(\Delta; X)$ in the following lemma.

\begin{lm}
 For a voter model process with parameter $A^*$, for any $\delta \in (0,1/2]$ and $\Delta \geq 0$, with probability at least $1-\delta$,
 $$
 |h(\Delta;X) - \E[h(\Delta;X)]| \leq \epsilon \|\Delta\|_F^2
 $$ 
where
$$
\epsilon = 2\sqrt{2}s\sqrt{m\left(\log\left(\frac{1}{2\pi^*}\right)+ \frac{1}{m}\log\left(\frac{2}{\delta}\right)\right)\frac{1}{\Phi_{A^*}}}\sqrt{\log\left(\frac{2}{\delta}\right)}.
$$
\label{lm:geom}
\end{lm}

\begin{proof}
Recall that
$$
h(\Delta; X) = \sum_{i=1}^m \sum_{t=0}^{\tau_i-1}\sum_{u=1}^n (\Delta_u^\top X_t^{(i)})^2.
$$
Let us define, if $0\leq s < \sum_{i=1}^m \tau_i$,
$$
X_s = X_{s-\sum_{i=1}^{k-1}\tau_i}^{(k)} \hbox{ for } \sum_{i=1}^{k-1}\tau_i \leq s < \sum_{i=1}^k \tau_i \hbox{ and } k\in [m]
$$
where $\sum_{i=1}^0 \tau_i = 0$ and, otherwise, if $s > \sum_{i=1}^m \tau_i$,
$$
X_{s} = 0.
$$

Now, we can write
$$
h(\Delta;X) = \sum_{t=0}^{\sum_{i=1}^m \tau_i-1} \sum_{u=1}^n (\Delta_u^\top X_t)^2.
$$

Let $Y_s = \sum_{u=1}^n (\Delta_u^\top X_s)^2$. For any $T > 0$, we have
\begin{eqnarray*}
&& \PP\left[|h(\Delta; X) - \E[h(\Delta; X)]| 
\geq  \epsilon \|\Delta\|_F^2 \right]\\
 &\leq & \PP\left[\left |\sum_{t=0}^{T-1} \sum_{u=1}^n (\Delta_u^\top X_t)^2 - \E\left[\sum_{t=0}^{T-1} \sum_{u=1}^n (\Delta_u^\top X_t)^2\right]\right | \geq \epsilon \|\Delta\|_F^2 \right] + \PP\left[\sum_{i=1}^m \tau_i > T\right] \\
 & = & \PP\left[\left |\sum_{s=0}^{T-1} Y_{s} - \EP\left [ \sum_{s=0}^{T-1} Y_{s}\right ]\right |\geq \epsilon \|\Delta\|_F^2 \right] + \PP\left[\sum_{i=1}^m \tau_i > T\right] \\
&\leq& \PP\left[\sum_{s=0}^{T-1} |Y_s-\E\left[Y_s \right]| \geq \epsilon \|\Delta\|_F^2\right]+\PP\left[\sum_{i=1}^m \tau_i > T\right]
\end{eqnarray*}
where the last inequality follows from the following basic relations
$$
\left |\sum_{s=1}^T Y_{s} - \E\left [ \sum_{s=1}^T Y_{s}\right ]\right |=\left |\sum_{s=1}^T (Y_{s} -   \E\left [Y_{s}\right])\right|\leq \sum_{s=1}^T |Y_{s} -  \E\left [Y_{s}\right]|.
$$

Let us define
$$
G_{t} = \sum_{s=1}^t |Y_s-\E\left[Y_s \right]| \hbox{ and }
\delta G_t = G_t - G_{t-1}.
$$

We have $\E[ G_t \mid \mathcal{F}_{t-1}] = G_{t-1} + \E\left[|Y_t-\E[Y_t]| \mid \mathcal{F}_{t-1}\right]\geq G_{t-1}$ for all $t\in \{1,\ldots,T\}$. Hence, $G_0,G_1,\ldots,G_T$ is a super-martingale sequence. Moreover, we have
\begin{eqnarray*}
Y_s &=& \sum_{u=1}^n (\Delta_u^\top X_s)^2 \\
&\leq& \sum_{u=1}^n \|\Delta_u\|_1^2 \\
&\leq& \|\Delta\|_{1,1}^2 \\
&\leq& s\|\Delta\|_F^2
\end{eqnarray*}
where the first inequality follows from $\Delta_{u,i}X_{s,i} \leq |\Delta_{u,i}|$ for all $i = 1,\ldots,n$, the second inequality follows from the fact $\sum_{u=1}^n \|\Delta_u\|_1^2 \leq \left(\sum_{u=1}^n \|\Delta_u\|_1\right)^2 = \|\Delta\|_{1,1}^2$ and the last inequality follows from the assumption that $\Delta$ has sparsity $s$, and, hence, $\|\Delta\|_{1,1} \leq \sqrt{s}\|\Delta\|_{2,2}$. 

It follows that $|\delta G_t|=|Y_t-\E[Y_t]|\leq 2s \|\Delta\|_F^2$ for all $t \in \{1,\ldots,T\}$, almost surely.

By applying Azuma-Hoeffding's inequality, we obtain
$$
\P\left[G_T-G_0 \geq \epsilon\right] = \P\left[\sum_{s=1}^T |Y_s-\E\left[Y_s \right]| \geq \epsilon \|\Delta\|_F^2\right]\leq  e^{-\frac{\epsilon^2}{8Ts^2}}.
$$
Hence, $\P\left[G_T-G_0 \geq \epsilon\right]\leq \delta/2$, for
$$
\epsilon \geq 2\sqrt{2}\sqrt{T}s\sqrt{\log\left(\frac{2}{\delta}\right)}.
$$

From Theorem~\ref{thm:tail}, we also have $\P\left[\sum_{i=1}^m \tau_i > T\right] \leq \delta/2$,
for any 
$$
T \geq m \left(\log\left(\frac{1}{2\pi^*}\right) + \frac{1}{m}\log\left(\frac{2}{\delta}\right)\right)\frac{1}{\Phi_{A^*}}.
$$

Hence, it follows that with probability at least $1-\delta$,
$$
|h(\Delta; X) - \E[h(\Delta; X)]|
\leq  \epsilon \|\Delta\|_F^2
$$
where
$$
\epsilon = 2\sqrt{2}s\sqrt{m\left(\log\left(\frac{1}{2\pi^*}\right)+ \frac{1}{m}\log\left(\frac{2}{\delta}\right)\right)\frac{1}{\Phi_{A^*}}}\sqrt{\log\left(\frac{2}{\delta}\right)}.
$$
\end{proof}

From Lemma~\ref{lm:geom}, it follows that with probability at least $1-\delta$, $h(\Delta;X) \geq  \E[h(\Delta;X)] - \epsilon ||\Delta||_F^2$. Combining with Lemma~\ref{lm:h}, with probability at least $1-\delta$, $h(\Delta;X) \geq  (\kappa_1- \epsilon) ||\Delta||_F^2$. Hence, with probability at least $1-\delta$, $h(\Delta;X) \geq  (\kappa_1/2) ||\Delta||_F^2$, provided that $\kappa_1/2\geq \epsilon$, which is equivalent to
$$
m \geq \frac{s^2 }{\Phi_{A^*}}\frac{1}{\EP[\tau]^2 \lambda_{\mathrm{min}}(\E[X_0 X_0^\top])^2}c_{\delta,\pi^*}(m).
$$

where 
$$
c_{\delta,\pi^*}(m) = 8\left(\log\left(\frac{1}{2\pi^*}\right)+\frac{1}{m}\log\left(\frac{2}{\delta}\right)\right)\log\left(\frac{2}{\delta}\right).
$$

\subsection{Proof of Lemma~\ref{lm:union}}

The proof of the lemma is based on using the concepts of covering and metric entropy which we discussed in Appendix~\ref{sec:covering}. With a slight abuse of notation, in the proof, we assume
$$
X = (X_0^{(1)}, \ldots, X_{\tau_1-1}^{(1)}, \ldots, X_0^{(m)}, \ldots, X_{\tau_m-1}^{(m)})^\top.
$$
The proof follows similar steps as that of Lemma~A.6 in \cite{SparseBer2019}. The main difference in our case is that $X$ does not have fixed dimensions, which requires additional technical steps. 

We separately consider three different cases: (a) $||\Delta||_F = r$, (b) $||\Delta||_F > r$, and (c) $||\Delta||_F < r$, for value of $r$ defined as
\begin{equation}
r = \frac{1}{\sqrt{s}}||A^*_{{S^*}^c}||_{1,1} + \ind_{\{||A^*_{{S^*}^c}||_{1,1} = 0\}}.
\label{equ:rvalue}
\end{equation}

\paragraph{Case 1: $||\Delta||_F = r$} For every $\Delta \in \mathbb{C}^*$, we have
\begin{eqnarray*}
||\Delta||_{1,1} &=& ||\Delta_{S^*}||_{1,1} + ||\Delta_{{S^*}^c}||_{1,1}
\\
&\leq & 4 ||\Delta_{S^*}||_{1,1} + 4 ||A^*_{{S^*}^c}||_{1,1}\\
&\leq & 4 (\sqrt{s}||\Delta ||_F + ||A^*_{{S^*}^c}||_{1,1}).
\end{eqnarray*}

Hence, for any $r > 0$, $\mathbb{C}^* \cap \partial B_2(r)\subseteq B_1(r')$
where 
\begin{equation}
r' := 4 (r\sqrt{s} + ||A^*_{{S^*}^c}||_{1,1}). 
\label{equ:rprime}
\end{equation}
Note that under (\ref{equ:rvalue}), $4\sqrt{s} \leq r'/r \leq 8\sqrt{s}$. 

By the triangle inequality, for all $\Delta,\Delta'\in \reals^{n\times n}$,
$$
\left|\sqrt{h(\Delta;X)}-\sqrt{h(\Delta'; X)}\right|\leq \sqrt{T}\rho(\Delta, \Delta')
$$
where
$$
\rho(\Delta,\Delta')= \frac{1}{\sqrt{T}}||X (\Delta-\Delta')^\top||_F.
$$

Because it holds $(a-b)^2 \geq a^2/2 - b^2$ for all $a,b\geq 0$, it follows that for all $\Delta, \Delta'\in \reals^{n\times n}$,
$$
h(\Delta;X)\geq \frac{1}{2}h(\Delta'; X) - T \rho(\Delta,\Delta')^2.
$$

Let $N$ be an $r\epsilon$-cover of $\mathbb{C}^*\cap \partial B_2(r)$. Fix an arbitrary $\Delta \in \mathbb{C}^*\cap \partial B_2(r)$, and let $
\Delta''$ be such that $\Delta''\in N$ and $\rho(\Delta,\Delta'')\leq r\epsilon$. Then, note
$$
h(\Delta;X) \geq \frac{1}{2}h(\Delta''; X) - T(r\epsilon)^2 \geq \frac{1}{2}\min_{\Delta'\in N} h(\Delta'; X) - T(r\epsilon)^2.
$$
It follows that
\begin{equation}
\inf_{\Delta \in\mathbb{C}^*\cap \partial B_2(r)} h(\Delta;X)\geq \frac{1}{2} \min_{\Delta \in N} h(\Delta; X) - T(r\epsilon)^2.
\label{equ:infh}
\end{equation}


Let $\epsilon^2 = (\kappa_1/8)/T$. It then follows
\begin{equation}
\PP\left[\inf_{\Delta \in\mathbb{C}^*\cap \partial B_2(r)} h(\Delta;X)  \leq  \frac{1}{8}\kappa_1 r^2\right] 
\leq \PP\left[\min_{\Delta \in N} h(\Delta;X)\leq \frac{1}{2}\kappa_1 r^2\right].
\label{equ:unionb}
\end{equation}

For any matrix $X$ with the number of rows $T$ such that $T\leq T^*$, for some fixed $T^*$, we can bound $|N|$ by a fixed value $N^*$, which we show next. Recall that for any $r > 0$, $\mathbb{C}^* \cap \partial B_2(r)\subseteq B_1(r')$, where $r'$ is defined in (\ref{equ:rprime}). Hence, we can bound $|N|$, the covering number of $\mathbb{C}^* \cap \partial B_2(r)$, by the covering number of $B_1(r')$. From Lemma~\ref{lm:n}, $r'/r\leq 8\sqrt{s}$ and $\epsilon^2 = (\kappa_1/8)/T$, we have 
$$
|N| \leq e^{c_1\left(\frac{r'}{r\epsilon}\right)^2\log(n)}\leq e^{ 8^3 s c_1 T \frac{1}{\kappa_1}\log(n)},
$$
where $c_1 = c\sigma_{\max}(X)^2/T$ for some constant $c > 0$, under condition $r\epsilon \leq c_2 r'$ with $c_2 = \sigma_{\max}(X)/\sqrt{T}$. Since $\sigma_{\max}(X)^2 \leq n T$ and $r'/r\geq 4\sqrt{s}$, we have
\begin{equation}
|N| \leq N^* = e^{c8^3 \frac{T^*}{\kappa_1}sn\log(n)}
\label{equ:Nstar}
\end{equation}
under conditions 
$$
C_1 = \{T\leq T^*\} \hbox{ and } C_2 = \left\{\frac{\sigma_{\max}(X)^2}{\kappa_1}\geq \frac{1}{128 s}\right\}.
$$

From (\ref{equ:unionb}), we have
\begin{eqnarray*}
&& \PP\left[\inf_{\Delta \in\mathbb{C}^*\cap \partial B_2(r)} h(\Delta;X)  \leq  \frac{1}{8}\kappa_1 r^2\right]\\
&\leq & \EP\left[\ind_{\cup_{\Delta\in N}\left\{h(\Delta;X)
\leq (\kappa_1/2)r^2\right\}}\ind_{C_1}\ind_{C_2}\right] + \EP[\ind_{C_1^c}] + \EP[\ind_{C_2^c}]\\
&\leq & N^* \max_{\Delta: ||\Delta||_F=r}\PP\left[h(\Delta;X)\leq \frac{1}{2}\kappa_1 r^2\right] + \PP[C_1^c] + \PP[C_2^c].
\end{eqnarray*}

The probability of the event $C_1^c$ can be bounded as follows. By Corollary~\ref{cor:sumtau}, for any $\delta' \in (0,1]$, $\Pr[C_1^c]=\Pr[T > T^*]\leq \delta'$, when
\begin{equation}
T^* \geq m \left(\log\left(\frac{1}{2\pi^*}\right)+\log\left(\frac{1}{\delta'}\right)\right)\frac{1}{\Phi_{A^*}}.
\label{equ:tstar}
\end{equation}

We next upper bound the probability of the event $C_2^c$. First, note
$$
\sigma_{\max}(X)^2 = \lambda_{\max}(X^\top X) \geq \frac{||X^\top X||_{1,1}}{n}
$$
where the last inequality holds by the basic fact that for any real $n\times n$ matrix $M$, $$
\lambda_{\max}(M)\geq \frac{\vec{1}^\top M \vec{1}}{\vec{1}^\top \vec{1}} = \frac{\sum_{i=1}^n \sum_{j=1}^n M_{i,j}}{n}.
$$
Hence, we have
\begin{equation}
\PP[C_2^c] \leq \PP\left[||X^\top X||_{1,1} < \frac{n}{128 s}\kappa_1\right].
\label{equ:prc2c}
\end{equation}
Note that
$$
||X^\top X||_{1,1} = \sum_{i=1}^m \sum_{t=0}^{\tau_i-1} \vec{1}^\top X_t^{(i)}{X_t^{(i)}}^\top \vec{1}.
$$

By the Palm inversion formula and definition of $\kappa_1$, we have
\begin{eqnarray*}
\EP[||X^\top X||_{1,1}] &=& m\EP\left[\sum_{t=0}^{\tau-1} \vec{1}^\top X_t {X_t}^\top\vec{1}\right]\\
&=& m \EP[\tau] \vec{1}^\top \E[X_0 X_0^\top] \vec{1}\\
&\geq & m \EP[\tau]\lambda_{\min}(\E[X_0 X_0^\top]) n\\
&=& n \kappa_1.
\end{eqnarray*}

To bound (\ref{equ:prc2c}), we use Chebyshev's inequality: for any random variable $X$ with expected value $\mu$ and variance $\sigma^2$, $\Pr[|X-\mu|\geq k\sigma]\leq 1/k^2$, for any $k>0$. By Chebyshev's inequality, we have
$$
\PP[||X^\top X||_{1,1} < x]\leq \frac{\sigma^2}{(\mu-x)^2} 
$$
for every $0 \leq x < \mu$, where $\mu$ and $\sigma^2$ are the expected value and variance of $||X^\top X||_{1,1}$, respectively. Let $\sigma_1^2$ be the variance of random variable $Y = \sum_{t=0}^{\tau-1} \vec{1}^\top X_t {X_t}^\top\vec{1}$
and $c_s = 1/(1-1/(128s))^2$. It follows
\begin{eqnarray*}
\PP\left[||X^\top X||_{1,1} < \frac{n}{128 s}\kappa_1\right] 
& \leq & c_s\frac{m\sigma_1^2}{n^2\kappa_1^2}\\
& = & c_s \frac{\sigma_1^2}{n^2 m \EP[\tau]^2 \lambda_{\min}(\E[X_0 X_0^\top])^2}.
\end{eqnarray*}

Next, we bound $\sigma_1^2$ as follows
$$
\sigma_1^2 \leq \EP[Y^2] \leq n^4 \EP[\tau^2] \leq \frac{2}{e^2} n^4 \frac{1}{\Phi_{A^*}^2}\frac{1}{\pi^*}
$$
where the last inequality is by Corollary~\ref{cor:mom}.

Putting the pieces together, we have $\Pr[C_2^c]\leq 1/n$, under condition
\begin{equation}
m\geq c_s n^3(1/\pi^*)\frac{1}{(\Phi_{A^*}\EP[\tau])^2\lambda_{\min}(\E[X_0 X_0^\top])^2}.
\label{equ:mcond3}
\end{equation}



By Lemma~\ref{lm:hrsc}, for any $\delta\in (0,1]$ and $\Delta$, $\Pr[h(\Delta;X)\leq (\kappa_1/2)||\Delta||_F^2] \leq \delta$ provided that
$$
m \geq 8 s^2\frac{\log(1/(2\pi^*))+\log(2/\delta)}{\Phi_{A^*}\EP[\tau]^2 \lambda_{\mathrm{min}}(\E[X_0 X_0^\top])^2}.
$$

Let $T^*$ be defined by equality in (\ref{equ:tstar}) and $\delta' = 1/n$. From (\ref{equ:Nstar}), we have $N^*= n^a$ where
$$
a = sn c8^3\frac{\log(1/(2\pi^*))+\log(n)}{(\Phi_{A^*}\EP[\tau])\lambda_{\min}(\E[X_0 X_0^\top])}.
$$

Take $\delta = 2/n^{a+1}$. Then, we have 
$$
N^* \max_{\Delta: ||\Delta||_F=r}\PP\left[h(\Delta;X)\leq \frac{1}{2}\kappa_1 r^2\right] \leq \frac{2}{n}
$$
provided that
\begin{equation}
m \geq 8 s^2\frac{\log(1/(2\pi^*))(a+1)\log(n) +  (a+1)^2\log(n)^2}{\Phi_{A^*}\EP[\tau]^2\lambda_{\min}(\E[X_0 X_0^\top])^2}.
\label{equ:mcond2}
\end{equation}

We have shown that the RSC condition holds with at least probability $1 - 4/n$, 
for all $\Delta \in\mathbb{C}^*\cap \partial B_2(r)$, with curvature $\kappa_1/8$ and tolerance $0$, under conditions (\ref{equ:mcond3}) and (\ref{equ:mcond2}). Condition (\ref{equ:mcond2}) is stronger than condition (\ref{equ:mcond3}) when 
$$
\EP[\tau]=\tilde{O}\left(\frac{s^4\pi^*}{n}\frac{1}{(\Phi_{A^*}\EP[\tau])\lambda_{\min}(\E[X_0 X_0^\top])}\right).
$$

\paragraph{Case 2: $||\Delta||_F > r$} Let $t = ||\Delta||_F / r$. Assume that the RSC condition holds for every $\Delta'\in \mathbb{C}^*\cap \partial B_2(r)$ with curvature $\kappa_1/8$ and tolerance $0$. Note that, for $\Delta \in \mathbb{C}^*\cap B_2(r)$,
$$
h(\Delta; X) = t^2 h(\Delta/t; X) \geq t^2 \frac{1}{8}\kappa_1 r^2 = \frac{1}{8}\kappa_1 ||\Delta||_F^2.
$$
Hence, the RSC condition holds on $\mathbb{C}^*\cap B_2(r)$ with curvature $\kappa
_1/8$ and tolerance $0$.

\paragraph{Case 3: $||\Delta||_F < r$} Let ${\gamma'}^2 = (\kappa_1/8) r^2 = (\kappa_1/8) ||A^*_{S^*}||_{1,1}^2/s$. In this case, we have
$$
h(\Delta;X)\geq 0\geq (\kappa_1/8)||\Delta||_F^2 - {\gamma'}^2.
$$
Hence, the RSC condition holds with curvature $\kappa_1/8$ and tolerance $(\kappa_1/8)||A^*_{S^*}||_{1,1}^2/s$. 

\subsection{Proof of Theorem \ref{thm:paramerr}}

The proof follows from Theorem~\ref{thm:mest} and Lemmas \ref{lm:reg}, \ref{lm:h0} and \ref{lm:union}, which we show as follows.

From Lemma~\ref{lm:reg}, with probability at least $1-1/n$,
$$
2||\nabla \ell(A^*)||_\infty \leq \lambda_m = 2\sqrt{2}\frac{1}{\alpha} 
\frac{1}{\sqrt{\Phi_{A^*}}}
c_{n,\pi^*} \sqrt{m}.
$$ 

From Lemma~\ref{lm:h0} and Lemma~\ref{lm:union}, the negative log-likelihood function satisfies the RSC condition relative to $A^*$ and $S$ that is the support of $A^*$ with curvature $\kappa' = \kappa_1/8$, and tolerance $\gamma'^2 = 0$
with probability at least $1-4/n$, under conditions (\ref{equ:m1}) and (\ref{equ:m2}).

%

Recall that $\kappa_1 = m \EP[\tau]\lambda_{\min}(\E[X_0 X_0^\top])$. It follows that 
$$
\frac{\lambda_m}{\kappa'} = \max\left\{\frac{8}{\E^0[\tau]\lambda_{\min}(\E[X_0X_0^\top])},1\right\}2\sqrt{2}\frac{1}{\alpha}\frac{1}{\sqrt{\Phi_{A^*}}}c_{n,\pi^*}\frac{1}{\sqrt{m}}.
$$

Combining the above facts with Theorem~\ref{thm:mest}, with probability at least $1-5/n$,
$$
||\hat{A}-A^*||_F^2  \leq   
4608\frac{s c_{n,\pi^*}^2}{\alpha^2 (\Phi_{A^*}\E^0[\tau])^2\lambda_{\min}(\E[X_0X_0^\top])^2}\Phi_{A^*}\frac{1}{m} 
$$
which completes the proof.

\subsection{Stationary correlation matrices}
\label{sec:corv}

We consider the stationary correlation matrix $M$ of the extended voter model (\ref{equ:ev}) defined by
$$
M = \E[X_0X_0^\top]. 
$$
The stationary correlation matrix exists which follows from the Palm inversion formula (\ref{equ:palm}) as $\EP[\tau] < \infty$. 

In this section we present analysis for the extended voter model that includes final consensus states of individual voter model processes. The stationary correlation matrix of this process, denoted as $M$, is related to the stationary correlation matrix, $M'$, for the extended voter model that does not include final consensus states of individual voter model processes as follows
\begin{equation}
M = \frac{\EP[\tau]}{\EP[\tau]+1}M' + \frac{\pi^\top \EP[X_0]}{\EP[\tau]}\vec{1}\vec{1}^\top. 
\label{equ:twocorr}
\end{equation}

This follows by the Palm inversion formula and the fact $\PP[X_{\tau} = \vec{1}] = \pi^\top \EP[X_0]$. Indeed,
\begin{eqnarray*}
(\EP[\tau]+1)M &=& \EP\left[\sum_{t=0}^{\tau}X_t X_t^\top\right]\\
&=& \EP\left[\sum_{t=0}^{\tau-1}X_t X_t^\top\right] + \EP[X_{\tau}X_{\tau}^\top]\\
&=& \EP[\tau]M' + \PP[X_{\tau}=\vec{1}]\vec{1}\vec{1}^\top \\
&=& \EP[\tau]M' + \pi^\top \EP[X_0] \vec{1}\vec{1}^\top.
\end{eqnarray*}

By Weyl's inequalities and the fact that eigenvalues of $\vec{1}\vec{1}^\top$ are either of value $0$ or $n$, the eigenvalues of $M$ and $M'$ are related as follows
$$
\frac{\EP[\tau]}{\EP[\tau]+1}\lambda_i(M') \leq \lambda_i(M) \leq \frac{\EP[\tau]}{\EP[\tau]+1}\lambda_i(M') + \frac{\pi^\top \EP[X_0]}{\EP[\tau]}n.
$$

\subsubsection{Lyapunov matrix equation}

The stationary correlation matrix satisfies the Lyapunov matrix equation stated in the following lemma. The Lyapunov matrix equation plays an important role for stability of linear dynamical systems \cite{GQ95,BS70}.

\begin{lm} For the extended voter model with parameter $A$ and initial state distribution $\mu$ such that the process is ergodic, the following Lyapunov matrix equation holds
\begin{equation}
M = A M A^\top + Q
\label{equ:lm}
\end{equation}
where 
$$
Q= \E[D(V(X_0))] + \frac{\EP[X_0X_0^\top] - (\mu(C_1) + (1-\mu(C))\pi^\top \EP[X_0\mid X_0\notin C]) \vec{1}\vec{1}^\top}{(1-\mu(C))\EP[\tau]+1}
$$
and $D(V(x))$ is the diagonal matrix with diagonal elements $V_{a_u}(x)$.
\label{lm:cor}
\end{lm}

\begin{proof} The following equations hold
\begin{eqnarray*}
\E[X_{t+1}X_{t+1}^\top] &=& \E[X_{t+1}X_{t+1}^\top \ind_{\{X_{t}\notin C\}}] \\
&& + \EP[X_0X_0^\top]\P[X_t\in C]
\end{eqnarray*}
\begin{eqnarray*}
\E[X_{t+1}X_{t+1}^\top \ind_{\{X_t\notin C\}}] &=& \E[Z_{t+1}X_tX_t^\top Z_{t+1}^\top] \\
&& - \E[Z_{t+1}X_tX_t^\top Z_{t+1}^\top \ind_{\{X_t\in C\}}]
\end{eqnarray*}

\begin{eqnarray*}
\E[Z_{t+1}X_tX_t^\top Z_{t+1}^\top] & = & \E[\E[Z_{t+1}X_t (Z_{t+1}X_t)^\top\mid X_t]]\\
&=& A\E[X_t X_t^\top]A^\top + \E[D(V(X_t))]
\end{eqnarray*}

and

$$
\E[Z_{t+1}X_tX_t^\top Z_{t+1}^\top \ind_{\{X_t\in C\}}] = \vec{1}\vec{1}^\top \P[X_t\in C_1].
$$

Putting the pieces together, we have

\begin{eqnarray*}
\E[X_0X_0^\top] &=& A\E[X_0X_0^\top]A^\top + \E[D(V(X_0))]\\
&& + (\EP[X_0X_0^\top] - p_1 \vec{1}\vec{1}^\top) \P[X_0 \in C]
\end{eqnarray*}
where $p_1 := \P[X_0\in C_1\mid X_0 \in C]$, from which the statement of the lemma follows.
\end{proof}

By multiplying both sides in equation (\ref{equ:lm}) with $\pi^\top$ and $\pi$ from left and right respectively, we obtain the following corollary.

\begin{cor} The following equation holds
\begin{eqnarray*}
&& ((1-\mu(C))\EP[\tau] + 1)\sum_{u=1}^n \pi_u^2 \E[V_{a_u}(X_0)]\\
&=& \EP[V_\pi(X_0)]+ \mu(C_1) - \mu(C) \EP[\pi^\top X_0].
\end{eqnarray*}
\end{cor}

If $\mu(C) = 0$, then the last expression boils down to
$$
(\EP[\tau] + 1) \sum_{u=1}^n \pi_u^2 \E[V_{a_u}(X_0)] = \EP[V_\pi(X_0)]
$$
which asserted in (\ref{equ:etaue}).

The necessary and sufficient condition for the Lyapunov matrix equation (\ref{equ:lm}) to have a unique solution $M$ for any positive semi-definite matrix $Q$ is that no two eigenvalues of $A$ have product equal to $1$, i.e. $\lambda_i(A)\lambda_j(A)\neq 1$ for all $i,j= 1,\ldots,n$. Furthermore, it is known that $\rho(A) < 1$, where $\rho(A)$ is the spectral radius of $A$, holds if and only if for any positive definite $Q$, (\ref{equ:lm}) has a a positive definite solution $M$. 

The Lyapunov matrix equation of the voter model is such that $Q$ is a positive semi-definite matrix as stated in the following lemma.

%
%

\begin{lm} $Q$ in (\ref{equ:lm}) is a positive semi-definite matrix, with eigenvalue $0$ associated with eigenvector $\pi$.
\label{lm:Qpi}
\end{lm}

\begin{proof} Multiply both sides of equation (\ref{equ:lm}) with $\pi^\top$ from the left and $\pi$ from the right. Note that $\pi^\top A M A^\top \pi = \pi^\top M \pi$. It follows that $\pi^\top Q \pi = 0$, which shows that $\pi$ is an eigenvector of $Q$ with eigenvalue $0$.
\end{proof}

\subsubsection{Product-form Bernoulli initial state distribution} We consider the spectrum of matrix $Q$ for initial state distribution $\mu$ that has product-form with Bernoulli ($p$) marginal distributions, with $0 < p < 1$. Note that 
$$
\EP[X_0\mid X_0\notin C] = \frac{1-(1-p)^{n-1}}{1-p^n -(1-p)^n}p\vec{1}.
$$
Under the given assumptions on distribution $\mu$, $Q$ is an $n\times n$ off-diagonal constant matrix with $q_{u,u} = \E[V_{a_u}(X_0)]$ and $q_{u,v} = -\alpha$, for $u\neq v$, where
\begin{equation}
\alpha: = \frac{p(1-p)}{(1-p^n - (1-p)^n)\EP[\tau] + 1}.
\label{equ:c}
\end{equation}


In order to localize eigenvalues of $Q$, we will use the following lemma.

\begin{lm}[\cite{G86}] Let $S$ be an $n\times n$ off-diagonal constant matrix such that $s_{i,j} = d_i + \alpha$, for $i=j$ and $s_{i,j} = \alpha$, for $i\neq j$, where $d_1, d_2, \ldots, d_n$ and $\alpha$ are given real numbers with $\alpha \geq 0$. Let $e_1  < \cdots < e_m$ be distinct values in $\{d_1, \ldots, d_n\}$, and $n_i$ be the number of occurrences of $e_i$. Then $S$ has
\begin{enumerate}
\item one eigenvalue in $(e_i, e_{i+1})$ for $i = 1,\ldots,m-1$ and one eigenvalue in $(e_m,\infty)$, all with multiplicity $1$;
\item each $e_i$ such that $n_i > 1$ is an eigenvalue of multiplicity $n_i-1$.
\end{enumerate}
\label{lm:gen}
\end{lm}

By Lemma \ref{lm:Qpi} and Lemma \ref{lm:gen}, matrix $Q$ in (\ref{equ:lm}) has eigenvalue $0$ and all other eigenvalues larger than or equal to $\min_{u}\E[V_{a_u}(X_0)] + \alpha > 0$. Hence, the smallest positive eigenvalue $\lambda_2(Q)$ of $Q$ satisfies
\begin{equation}
\lambda_2(Q) \geq \min_{u\in V} \E[V_{a_u}(X_0)] + \alpha > 0.
\label{equ:l2Q}
\end{equation}


\subsubsection{Bounding smallest eigenvalue of the stationary correlation matrix} Let $R(S,x)$ denote the Rayleigh quotient, $R(S,x) = (x^\top S x)/x^\top x$, for some matrix $S$. Note that
$$
R(M;\pi) = \frac{\E[(\pi^\top X_0)^2]}{||\pi||^2}.
$$
Note that
$$
\lambda_2(Q) = \min_{x: x\neq 0, \pi^\top x = 0} R(Q,x).
$$

We claim that
$$
\lambda_1(M) \geq \min\{R(M; \pi), \lambda_2(Q)\}.
$$

To show this, we decompose any vector $x\in \reals^n$ into orthogonal components $x = \gamma \pi + z$ for some $\gamma\in \reals$ and $z\in \reals^n$ such that $\pi^\top z=0$. We have the following relations
\begin{eqnarray*}
R(M,x) &=& \frac{x^\top M x}{x^\top x}\\
 &=& \frac{(\gamma \pi + z)^\top M (\gamma \pi + z)}{(\gamma \pi + z)^\top (\gamma\pi + z)}\\
&=& \frac{(\gamma \pi + z)^\top M (\gamma \pi + z)}{\gamma^2 \pi^\top \pi + z^\top z}\\
&\geq & \frac{\gamma^2 \pi^\top M \pi + z^\top M z}{\gamma^2 \pi^\top \pi + z^\top z}\\
&\geq & \frac{\gamma^2 \pi^\top \pi R(M;\pi) + z^\top z \lambda_2(Q)}{\alpha^2 \pi^\top \pi + z^\top z}\\
&\geq & \min\{R(M;\pi), \lambda_2(Q)\}.
\end{eqnarray*}
From this it follows that the smallest eigenvalue $\lambda_1(M)$ of $M$ statisfies $\lambda_1(M) = \min_{x: x\neq 0} R(M,x) \geq \min\{R(M;\pi),\lambda_2(Q)\}$.

It readily follows from (\ref{equ:ev}) that the Rayleigh quotient $R(M;\pi)$ can be lower bounded as follows
$$
R(M;\pi) \geq \frac{1}{(1-\mu(C))\E[\tau]+1}\frac{\EP[(\pi^\top X_0)^2]}{||\pi||^2}.
$$

For the case when $\mu$ is a product-form distribution with Bernoulli ($p$) marginal distributions, we have
$$
\EP[(\pi^\top X_0)^2]= p^2 + p(1-p) ||\pi||^2.
$$
Hence, $R(M;\pi) \geq p(1-p + p/||\pi||^2)/(\EP[\tau] + 1)$. By combining with (\ref{equ:l2Q}), we have
$$
\lambda_1(M) \geq \min\left\{\frac{p\left(1-p+p\frac{1}{||\pi||^2}\right)}{{\EP[\tau] + 1}}, \min_{u\in V}\E[V_{a_u}(X_0)] + \frac{p(1-p)}{{\EP[\tau] + 1}}\right\}.
$$

\subsubsection{Complete graph example} We consider the case when $a_{u,u} = 0$ and $a_{u,v} = 1/(n-1)$ for all $u\neq v$, and when the initial state distribution $\mu$ is the product-form with Bernoulli ($p$) marginal distributions. Because of the symmetry, $m_{u,u} = a$ and $m_{u,v} = b$, for all $u\neq v$, for some $a$ and $b$. Note
\begin{eqnarray*}
(AMA^\top)_{u,v} &=& \sum_{i=1}^n \sum_{j=1}^n a_{u,i}a_{v,j} m_{i,j}\\
&=& a \sum_{i=1}^n a_{u,i}a_{v,i} + b\sum_{i=1}^n \sum_{j\neq i} a_{u,i}a_{v,j} \\
&=& (a-b)\sum_{i=1}^n a_{u,i}a_{v,i} + b \sum_{i=1}^n \sum_{j=1}^n a_{u,i}a_{v,j}.
\end{eqnarray*}
It follows 
$$
(AMA^\top)_{u,v} = (a-b) \left(\frac{1}{n-1}\ind_{\{u=v\}} + \frac{n-2}{(n-1)^2}\ind_{\{u\neq v\}}\right) + b.
$$
Hence, we have
$$
(M-AMA^\top)_{u,v} =  \left\{
\begin{array}{ll}
(a-b)\left(1-\frac{1}{n-1}\right) & \hbox{ if } u = v\\
-(a-b)\frac{n-2}{(n-1)^2} & \hbox{ if } u \neq v.
\end{array}
\right .
$$
Now, we have $M - A M A^\top = Q$, and $Q$ has diagonal elements of value $\E[V_{a_1}(X_0)]$ and off-diagonal elements of value $-\alpha$, where $\alpha$ is given in (\ref{equ:c}). It follows that
$$
a - b = \frac{(n-1)^2}{n-2} \frac{1}{(1-p^n-(1-p)^n)\EP[\tau] + 1} p(1-p)
$$
and
$$
\E[V_{a_1}(X_0)] = \left(1-\frac{1}{n-1}\right)(a-b).
$$
Note that 
$$
a-b = \frac{n p(1-p)}{\EP[\tau]} (1+o(1))
$$
and
$$
\E[V_{a_1}(X_0)] = p(1-p)\frac{n}{\EP[\tau]+1} (1+o(1)).
$$

From Lemma~\ref{lm:gen}, it follows that $\lambda_1(M)=a-b$. Hence, 
$$
\lambda_1(M) = p(1-p)\frac{n}{\EP[\tau]+1}(1+o(1)).
$$

The smallest eigenvalue of the correlation matrix $K$ with respect to the stationary distribution of an extended voter process that includes final consensus states of individual voter model processes, $\lambda_1(K)$, such that $K$ has constant diagonal elements and constant non-diagonal elements, $\lambda_1(K)$, is related to the smallest eigenvalue of the correlation matrix $K'$ with respect to stationary distribution of the extended voter process that does not include final consensus states, $\lambda_1(K')$, as follows:
\begin{equation}
\lambda_1(K) = \frac{\EP[\tau]}{\EP[\tau]+1}\lambda_1(K').
\label{equ:lambda1m}
\end{equation}
This easily follows from (\ref{equ:twocorr}) and the fact that the smallest eigenvalue of a matrix with constant diagonal elements (say equal $\alpha$) and constant non-diagonal elements (say equal to $\beta$) is equal to $\alpha - \beta$. 

For the complete graph case considered in this section, from (\ref{equ:lambda1m}) and $\EP[\tau]/(\EP[\tau]+1) = 1 + o(1)$, the correlation matrix $M'$ with respect to stationary distribution of the extended voter process that does not include final consensus states of individual voter model processes is
$$
\lambda_1(M') = p(1-p)\frac{n}{\EP[\tau]+1}(1+o(1)).
$$

\subsection{Linear $\epsilon$-noisy voter model}
\label{sec:corev}

In this section we consider the linear $\epsilon$-voter model defined as the linear discrete-time dynamical system (\ref{equ:xqzxr}). Let $A^\epsilon: = (1-2\epsilon)A$. 

\begin{lm} For any linear $\epsilon$-voter model with parameters $A$ and $0< \epsilon \leq 1/2$, the expected values of vertex states satisfy
$$
\E[X_0] = \frac{1}{2}\vec{1}.
$$
\label{lm:exnoisy}
\end{lm}

\begin{proof} From (\ref{equ:xqzxr}), we have
$$
\E[X_{t+1}] = \E[Q_{t+1}Z_{t+1}X_t] + \E[R_t] = (1-2\epsilon)A \E[X_t] +\epsilon \vec{1}.
$$
Hence, 
$$
(I-(1-2\epsilon)A)\E[X_0] = \epsilon\vec{1}.
$$
Since $\rho((1-2\epsilon) A) = 1-2\epsilon < 1$, we have
$$
\E[X_0] = (I-(1-2\epsilon) A)^{-1}\epsilon \vec{1}.
$$
Now, note
\begin{eqnarray*}
(I-(1-2\epsilon)A)^{-1}\vec{1} &=& \sum_{i=0}^\infty (1-2\epsilon)^i A^i\vec{1}\\
&=& \sum_{i=0}^\infty (1-2\epsilon)^i \vec{1}\\
&=& \frac{1}{2\epsilon}\vec{1}.
\end{eqnarray*}
Hence, it holds $\E[X_0]= (1/2)\vec{1}$.
\end{proof}

\begin{lm} For any linear $\epsilon$-voter model with parameters $A$ and $0< \epsilon \leq 1/2$, the following Lyapunov matrix equation holds
$$
M = A^\epsilon M {A^\epsilon}^\top + Q^\epsilon
$$
where
$$
\E[M]= \E[X_0X_0^\top]
$$
and
$$
Q^\epsilon = \E[D(V^\epsilon(X_0))] + \epsilon^2 I + \epsilon(1-\epsilon)\vec{1}\vec{1}^\top.
$$
\label{lm:lypnoisy}
\end{lm}

\begin{proof} From (\ref{equ:xqzxr}), we have
\begin{eqnarray*}
X_{t+1}X_{t+1}^\top &=& (D(Q_{t+1})Z_{t+1}X_t + R_t)(D(Q_{t+1})Z_{t+1}X_t + R_t)^\top\\
&=& (D(Q_{t+1})Z_{t+1}X_t + R_t)((D(Q_{t+1})Z_{t+1}X_t)^\top + R_t^\top)\\
&=& D(Q_{t+1})Z_{t+1}X_t X_t^\top (D(Q_{t+1})Z_{t+1})^\top + D(Q_{t+1})Z_{t+1}X_t R_t^\top\\
&& + (D(Q_{t+1})Z_{t+1}X_t R_t^\top)^\top + R_t R_t^{\top}.
\end{eqnarray*}

Now, note
\begin{eqnarray*}
&& \E[(D(Q_{t+1})Z_{t+1}X_t + R_t)((D(Q_{t+1})Z_{t+1}X_t)^\top+R_t^\top)]\\
& =& A^\epsilon \E[X_tX_t^\top] {A^\epsilon}^\top + \E[D(V^\epsilon(X_t))].
\end{eqnarray*}

$$
(D(Q_{t+1})Z_{t+1}X_t R_t^\top)_{u,v} = Q_{t+1,u}\sum_w Z_{t+1,u,w} X_{t,w} R_{t,v}
$$

$$
\E[(D(Q_{t+1})Z_{t+1}X_t R_t^\top)_{u,v}] = \epsilon(1-2\epsilon) a_{u}^\top \E[X_t] \ind_{\{u\neq v\}}
$$

$$
\E[D(Q_{t+1})Z_{t+1}X_t R_t^\top] = \epsilon D(A^\epsilon \E[X_t])(\vec{1}\vec{1}^\top - I)
$$

and 
$$
\E[R_t R_t^\top] = \epsilon((1-\epsilon)I + \epsilon \vec{1}\vec{1}^\top).
$$

Putting the pieces together, we have
\begin{eqnarray*}
Q^\epsilon &=& \E[D(V^\epsilon(X_0))] + \\
&& + \epsilon(D(A^\epsilon\E[X_0])(\vec{1}\vec{1}^\top - I) + (\vec{1}\vec{1}^\top - I)D(A^\epsilon\E[X_0]))\\
&& + \epsilon(1-\epsilon)I + \epsilon^2\vec{1}\vec{1}^\top.
\end{eqnarray*}

Since, by Lemma~\ref{lm:exnoisy}, $\E[X_0] = (1/2)\vec{1}$, we have
\begin{eqnarray*}
&& \epsilon(D(A^\epsilon\E[X_0])(\vec{1}\vec{1}^\top - I) + (\vec{1}\vec{1}^\top - I)D(A^\epsilon\E[X_0])) +\epsilon((1-\epsilon)I + \epsilon\vec{1}\vec{1}^\top)\\
&=& \epsilon(1-2\epsilon)(\vec{1}\vec{1}^\top - I) + \epsilon(1-\epsilon)I + \epsilon^2\vec{1}\vec{1}^\top\\
&=& \epsilon^2 I + \epsilon(1-\epsilon)\vec{1}\vec{1}^\top.
\end{eqnarray*}
\end{proof}

\begin{lm} For any linear $\epsilon$-voter model with parameters $A$ and $0<\epsilon \leq 1/2$, we have
$$
\lambda_{\mathrm{min}}(\E[X_0X_0^\top]) \geq \frac{1}{1-\lambda_{\mathrm{min}}(A)^2}(\min_u \E[V_{a_u^\epsilon}(X_0)] + \epsilon^2) \geq \epsilon^2.
$$
\end{lm}

\begin{proof} Let $M$ and $Q^\epsilon$ be defined as in Lemma~\ref{lm:lypnoisy}. It is known that (see, e.g. \cite{YH79}),
$$
\lambda_{\mathrm{min}}(M)\geq \frac{1}{1-\lambda_{\mathrm{min}}(A)^2}\lambda_{\mathrm{min}}(Q^\epsilon).
$$
Note that $Q^\epsilon$ is a matrix with constant off-diagonal elements equal to $\alpha := \epsilon(1-\epsilon)$ and diagonal elements equal to 
$$
\E[V_{a_u^\epsilon}(X_0)] + \epsilon^2 + \alpha. 
$$

By Lemma~\ref{lm:gen}, it follows
$$
\lambda_{\mathrm{min}}(Q^\epsilon) \geq \min_u \E[V_{a_u^\epsilon}(X_0)] + \epsilon^2.
$$
\end{proof}

\subsection{Proof of Theorem~\ref{thm:lb}}  The proof follows similar steps as that for linear dynamical systems with additive Gaussian noise  \cite{JP19}. The differences lie in steps that are needed to resolve technical points that arise due to Bernoulli random variables and underlying constraints on the model parameter.

For any substochastic $n\times n$ matrix $A$, $x_0,\ldots, x_t \in \{0,1\}^n$, and $t\geq 0$, let us define
$$
p_A(x_0, \ldots, x_t) := \P_A[X_0=x_0, \ldots, X_t=x_t]
$$
and, for any $x,y\in \{0,1\}^n$,
$$
p_A(y\mid x) := \P_A[X_{t+1}=y\mid X_t = x].
$$

Let $A^*$ be an $n\times n$ stochastic matrix and $A$ be an $n\times n$ substochastic matrix such that $A\neq A^*$. The log-likelihood ratio of the observed voter model process states, under parameters $A^*$ and $A$, is given by
$$
L(X) = \sum_{i=1}^m \log\left(\frac{p_{A^*}(X_0^{(i)}, \ldots, X_{\tau_i}^{(i)})}{p_{A}(X_0^{(i)}, \ldots, X_{\tau_i}^{(i)})}\right).
$$

For every $t\geq 1$, we have
$$
p_{A^*}(x_0, \ldots, x_t) = \mu(x_0) p_{A^*}(x_1\mid x_{0})\cdots p_{A^*}(x_t\mid x_{t-1}).
$$

Now, note 
\begin{eqnarray*}
\EP_{A^*}[L(X)] &=& m \EP_{A^*}\left[\sum_{t=0}^{\tau-1} \EP_{A^*}\left[\log\left(\frac{p_{A^*}(X_{t+1}\mid X_{t})}{p_{A}(X_{t+1}\mid X_{t})}\right)\mid {\mathcal F}_{t}\right]\right]\\\nonumber
&=& m \EP_{A^*}\left[\sum_{t=0}^{\tau -1} \sum_{u=1}^n \mathrm{KL}({a_u^*}^\top X_t \mid\mid {a_u}^\top X_t)\right].\label{equ:EAL}
\end{eqnarray*}

For every $u\in V$, let $S_u$ denote the support of $a_u^*$, $\mathcal{X}_u = \{x\in \{0,1\}^n: 0<\sum_{j\in S_u}x_j < |S_u|\}$, and $\mathcal{A}^*_u =\{a\in \reals_+^n \mid \sum_{v=1}^n a_v = 1, a_w \geq \alpha/2, \hbox{ for all } w\in S_u\}$. Then, for every $x\in \mathcal{X}_u$ and $a_u \in \mathcal{A}^*_u$, we have $\alpha/2 \leq x^\top a_u \leq 1-\alpha/2$. 

If $X_s\in \mathcal{X}_u$ and $a_u\in \mathcal{A}_u^*$, we have 
\begin{eqnarray*}
\mathrm{KL}({a_u^*}^\top X_t \mid\mid {a_u}^\top X_t) & \leq & \frac{4}{\alpha}({a_u^*}^\top X_t-{a_u}^\top X_t)^2\\
&=& \frac{4}{\alpha}X_t^\top (a^*_u-a_u)(a^*_u-a_u)^\top X_t.
\end{eqnarray*}

In the remainder of the proof, we assume that $A$ is such that $a_u\in \mathcal{A}^*_u$ for all $u\in \{1,\ldots,n\}$. It follows
\begin{eqnarray*}
\EP_{A^*}[L(X)] & \leq & m\frac{4}{\alpha}\EP_{A^*}\left[\sum_{t=0}^{\tau-1} \sum_{u=1}^n ({a_u^*}^\top X_t - {a_u}^\top X_t)^2\right]\\
&=& m\frac{4}{\alpha}\EP_{A^*}\left[\sum_{t=0}^{\tau-1} \sum_{u=1}^n \EP_{A^*}[X_t^\top ({a_u}^*-a_u)(a_u^*-a_u)^\top X_t]\right]\\
&=& m\frac{4}{\alpha}\EP_{A^*}\left[\sum_{t=0}^{\tau-1} X_t^\top W X_t\right]
\end{eqnarray*}
where $W:=(A^*-A)^\top (A^*-A)$.

By the elementary properties of the trace of a matrix, we have
$$
\EP_{A^*}\left[\sum_{t=0}^{\tau-1} X_t^\top W X_t\right] 
= \mathrm{tr}\left(W \EP_{A^*}\left[\sum_{t=0}^{\tau-1} X_tX_t^\top\right]\right).
$$

Hence, it holds
$$
\EP_{A^*}[L(X)] \leq m\frac{4}{\alpha} \mathrm{tr}\left(W \EP_{A^*}\left[\sum_{t=0}^{\tau-1} X_tX_t^\top\right]\right).
$$

By the Palm inversion formula (\ref{equ:palm}), we have
$$
\EP_{A^*}\left[\sum_{t=0}^{\tau-1} X_tX_t^\top\right] = \EP_{A^*}[\tau]\E_{A^*}[X_0X_0^\top].
$$

It follows that 
\begin{equation}
\EP_{A^*}[L(X)] \leq \frac{4}{\alpha}m \EP_{A^*}[\tau]\mathrm{tr}(W \E_{A^*}[X_0X_0^\top]).
\label{equ:elm}
\end{equation}

Let $\mathcal{F}_m$ denote the $\sigma$-algebra of observations from $m$ independent realizations of the voter model process with parameter $A^*$ and initial state distribution $\mu$. By the data processing inequality, we have
\begin{equation}
\EP_{A^*}[L(X)] \geq \sup_{E \in \mathcal{F}_m} \kl(\PP_{A^*}[E]\mid\mid \PP_{A}[E]).
\label{equ:dpi}
\end{equation}

Assume in addition that $A$ satisfies $2\epsilon\leq ||A-A^*||_F\leq 3\epsilon$, and assume that $m\geq m_0$. Let $E$ be the $\mathcal{F}_m$-measurable event defined as
$$
E=\{||\hat{A}-A^*||_F\leq \epsilon\}.
$$

Since the algorithm is $(\epsilon,\delta)$-locally stable, we have
$$
\P_{A^*}[||\hat{A}-A^*||_F\leq \epsilon]\geq 1-\delta
$$
and
$$
\P_{A}[||\hat{A}-A^*||_F\leq \epsilon] \leq \P_{A}[||\hat{A}-A||_F > \epsilon]\leq \delta.
$$

Hence, it follows
\begin{equation}
\kl(\PP_{A^*}(E)\mid\mid \PP_{A}(E))\geq \kl(1-\delta\mid\mid \delta) 
\geq \log\left(\frac{1}{2.4\delta}\right).
\label{equ:kllbx}
\end{equation}

Combining (\ref{equ:elm}), (\ref{equ:dpi}) and (\ref{equ:kllbx}), it follows that for any $(\epsilon,\delta)$-locally stable estimator in $A^*$, for all substochastic matrices $A$ such that (a) $2\epsilon \leq ||A-A^*||_F\leq 3\epsilon$ and (b)
$a_{u,v}\geq \alpha/2$ for every $(u,v)$ in the support of $A^*$, and $m\geq m_0$, we have
\begin{equation}
m \EP_{A^*}[\tau] \mathrm{tr}\left(W \E_{A^*}[X_0X_0^\top]\right) \geq \frac{\alpha}{4}\log\left(\frac{1}{2.4\delta}\right)
\label{equ:condW}
\end{equation}
where, recall,
$$
W = (A^*-A)^\top (A^*-A).
$$
We need to show that there exists a substochastic matrix $A$ that minimizes the left-hand side of inequality (\ref{equ:condW}) under the given constraints. 

Let $\E_{A^*}[X_0 X_0^\top] = Q \Lambda Q^\top$ be the eigenvalue decomposition of the correlation matrix $\E_{A^*}[X_0 X_0^\top]$, with eigenvalues $\lambda_1\leq \cdots \leq \lambda_n$, and eigenvectors $Q = (q_1, \ldots, q_n)$. Hence, we have $\E_{A^*}[X_0 X_0^\top] = \sum_{i=1}^n \lambda_i q_iq_i^\top$. Finding $A$ that minimizes the left-hand side of the inequality (\ref{equ:condW}) corresponds to finding $A$ that is a solution of the following optimization problem:
$$
\begin{array}{rl}
\hbox{minimize} & \mathrm{tr}(W\sum_{v=1}^n \lambda_v q_v q_v^\top)\\
\hbox{subject to} & W = (A^*-A)^\top (A^*-A)\\
& 2\epsilon \leq \sqrt{\mathrm{tr}(W)}\leq 3\epsilon\\
& a_{u,v}\geq \alpha/2 \hbox{ for every } (u,v) \hbox { in the support of } A^*\\
& A \hbox{ is a substochastic matrix}.
\end{array}
$$

By taking $W = 4\epsilon^2 q_1 q_1^\top$, we have 
$$
\mathrm{tr}(W\E_{A^*}[X_0X_0^\top])=4\epsilon^2 \lambda_1
$$
and 
$$||A-A^*||_F = \sqrt{\mathrm{tr}(W)} = 2\epsilon.
$$ 

We need to show that there exists a matrix $A$ that  satisfies the constraints of the above optimization problem. To show this, let $A$ be such that for some fixed $u\in \{1,\ldots,n\}$, $a_{v}^* = a_v$ for all $v\neq u$, and $a_u$ is given by
\begin{equation}
a_u^* - a_u = 2\epsilon q_1.
\label{equ:aup}
\end{equation}

For such a matrix $A$, we have $W = (a_u^*-a_u) (a_u^*-a_u)^\top$, and clearly $W = 4\epsilon^2 q_1 q_1^\top$. Note that $||A^*-A||_F = ||a_u^*-a_u||_2 = 2\epsilon$.


From (\ref{equ:aup}), we have
$$
a_u^\top \vec{1} = 1 - 2\epsilon q_1^\top \vec{1}.
$$ 
Hence, $A$ is a substochastic matrix, if and only if, 
$$
2\epsilon |q_1^\top \vec{1}| \leq 1.
$$

By Cauchy-Schwartz inequality $|q_1^\top \vec{1}| \leq ||\vec{1}||_2 ||q_1||_2 =\sqrt{n}$. Hence, if $\epsilon \leq 1/(2\sqrt{n})$, then $A$ is a substochastic matrix.



From (\ref{equ:aup}), for every $(u,v)$ which is in the support of $A^*$, i.e. $a^*_{u,v} > 0$, it must hold
$$
a_{u,v}^* - 2\epsilon q_{1,v} \geq \alpha/2. 
$$
Since $|q_{1,v}| \leq 1$ for all $v$, $a_{u,v}^*\geq \alpha$ for all $(u,v)$ in the support of $A^*$, we have that the above condition holds if $\epsilon \leq \alpha/4$.

\section{Asynchronous voter model on a path}
\label{sec:path}

We consider the asynchronous voter model process on a path of two or more vertices, where at each time step one vertex, chosen uniformly at random, updates its state. We assume that initial node states are such that $k$ vertices on one end of the path are in state $1$ and other vertices are in state $0$. For any such initial state, at every time step, there are at most two vertices with a mixed neighborhood set. If such two vertices exist, they reside on the boundary separating the state-$1$ vertices from state-$0$ vertices. Note that an \emph{informative interaction} for the parameter estimation problem occurs only when one of these boundary vertices samples a neighbour. See Figure~\ref{fig:path} for an illustration. 

\begin{figure}[t]
\centering
\includegraphics[width=0.98\linewidth]{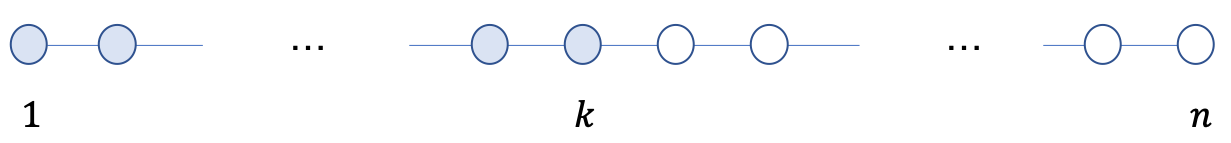}     
\caption{A path example: initial state is such that $k$ leftmost vertices are in state $1$ and the remaining $n-k$ rightmost vertices are in state $0$. Each vertex samples a neighbour equiprobably.}
\label{fig:path}
\end{figure}

The expected number of vertices that perform at least one informative interaction until the voter model process hits a consensus state can be characterized as asserted in the following proposition.

\begin{pr} Consider the voter model on a path with $n \geq 2$ vertices such that each vertex samples a neighbor with equal probabilities, with initial states such that $k$ vertices on one end of the path are in state $1$ and other vertices are in state $0$. Then, the number of vertices $N$ that participate in at least one informative interaction has the expected value 
$$
\E[N] = k\log\left(\frac{n}{k}\right) + (n-k)\log\left(\frac{n}{n-k+1}\right) + \Theta(1).
$$
\end{pr}

Note that if $k/n = o(1)$, then
$$
\E[N] = k \left(\log\left(\frac{n}{k}\right) + \Theta(1)\right).
$$

On the other hand, if $k$ is a fixed constant, then 
$$
\E[N] = O(\log(n)).
$$

This makes precise the intuition that only a small fraction of vertices will participate in at least one informative interaction if a small fraction of vertices on one end of the path are in state $1$ and other vertices are in state $0$, for asymptotically large $n$.

In the remainder of this section, we prove the proposition. Let $I_t$ denote the vertex activated at time step $t$. Let $p_u$ denotes the probability with which vertex $u$ samples vertex $u+1$. Then, $1-p_u$ is the probability of vertex $u$ sampling vertex $u-1$. Here $p=(p_1, p_2, \ldots, p_n)$ are parameters such that $p_u \in (0,1)$ for $u = 1,2,\ldots,n$. In the proposition, we consider the case $p_1 = 1$, $p_n = 0$, and $p_u = 1/2$, for $u\in \{1,\ldots, n\}\setminus \{1,n\}$.

If vertex $u$ is active at time step $t$, then the state of $u$ is according to
$$
\P[X_{t+1,u} = 1 \mid X_t = x] =
\left\{
\begin{array}{ll}
(1-p_1)x_1 + p_1 x_2 & \hbox{ if } u =1\\
(1-p_u)x_{u-1} + p_u x_{u+1} & \hbox{ if } 1 < u < n,\\
(1-p_n)x_{n-1} + p_n x_n & \hbox{ if } u = n,
\end{array}
\right .
$$
and the states of other vertices remain unchanged. 
 
Given observed node states until absorption to a consensus state, we want to estimate the values of parameters $p_1, \ldots, p_{n}$. We consider this parameter estimation problem for the initial state such that $X_{0,u} = 1$ for $u\in \{1,\ldots,k\}$ and $X_{0,u} = 0$ for $u\in \{k+1,\ldots, n\}$, for some $k \in \{1, \ldots, n-1\}$.

Under given condition on the initial state, the system dynamics is fully described by $Y_t$ defined as the number of vertices  in state $1$ at time step $t$. Note that $\{Y_t\}_{t\geq 0}$ is a Markov chain with state space $\{0,\ldots,n\}$, initial state $Y_0=k$ and the transition probabilities:
\begin{eqnarray*}
\P[Y_{t+1} = y+1 \mid Y_t = y] &=& \frac{1}{n} (1-p_{y+1}) \ind_{\{0 < y< n\}}\label{equ:mc1}\\
\P[Y_t = y-1 \mid Y_t = y] &=& \frac{1}{n} p_{y} \ind_{\{0 < y < n\}}\label{equ:mc2}\\
\P[Y_{t+1} = y \mid Y_t = y] &=& 1-\frac{1}{n} (1-p_{y+1}) \ind_{\{0< y< n\}} - \frac{1}{n} p_{y} \ind_{\{0< y < n\}}.\label{equ:mc3}
\end{eqnarray*}
This Markov chain has two absorbing states $0$ and $n$.

The Markov chain $\{Y_t\}_{t\geq 0}$ has a jump point at time step $t$ if, and only if, (a) vertex $Y_t$ is active and this vertex samples vertex $Y_t+1$ or (b) vertex $Y_t+1$ is active and this vertex samples vertex $Y_t$. We refer to each jump point of $Y$ as a useful interaction as only at a jump point we can observe outcome of a Bernoulli experiment of which vertex is sampled by the active node, with parameter in a strict interior of $(0,1)$. 

If at time step $t$, the active vertex is $I_t = Y_t$, then this vertex sampled vertex $Y_t+1$ if we observe $Y_{t+1}-Y_t = -1$, otherwise this vertex sampled vertex $Y_t-1$ if we observe $Y_{t+1}-Y_t = 0$. If at time step $t$, the active vertex is $I_t = Y_{t+1}+1$, then this vertex sampled vertex $Y_t$ if we observe $Y_{t+1}-Y_t = 1$ and otherwise this vertex sampled $Y_t+2$ if we observe $Y_{t+1}-Y_t = 0$.

We next consider the probability that a given vertex has at least one informative interaction before $Y$ gets absorbed in either state $0$ or $1$. This is of interest because the maximum likelihood estimate of $p_u$ is well defined only if at least one informative interaction is performed for vertex $u$. 

Let $h_{u,y}$ be the probability that vertex $u$ has at least one informative interaction given that $Y$ started at initial state $Y_0 = y$. We want to compute the values of $h_{u,k}$ for $1\leq u\leq n$ and $1<k<n$. 

\paragraph{Case $k< u$} Vertex $u$ has at least one informative interaction if, and only if, there exists a time step $t$ such that $I_t = u$ and $Y_t = u-1$. Let $\tilde{Y}$ be a Markov chain with state space $\{1,2,\ldots,u\}$ and transition probabilities for vertices  $0\leq v < u-1$ corresponding to those of $Y$ and by definition 
\begin{eqnarray*}
\P[\tilde{Y}_{t+1} = u \mid \tilde{Y}_t = u-1] &=& \frac{1}{n}\\
\P[\tilde{Y}_{t+1} = u-2\mid \tilde{Y}_t = u-1] &=& \frac{1}{n}p_{u-1}\\
\P[\tilde{Y}_{t+1} = u-1 \mid \tilde{Y}_t = u-1] &=& 1 - \frac{1}{n} - \frac{1}{n}p_{u-1}
\end{eqnarray*}
and
$$
\P[\tilde{Y}_{t+1} = u \mid \tilde{Y}_t = u] = 1.
$$

The value of $h_{u,k}$ corresponds to the probability of Markov chain $\tilde{Y}$ hitting state $u$ by starting from state $k$. We have boundary conditions $h_{u,0} = 0$ and $h_{u,u} = 1$. By the first-step analysis of Markov chains, for $0 < y < u$,
$$
h_{u,y}= \frac{1}{n}p_y h_{u,y-1} + \frac{1}{n}(1-p_{y+1})h_{u,y+1} + \left(1-\frac{1}{n}p_y - \frac{1}{n}(1-p_{y+1})\right)h_{u,y}
$$
where we abuse the notation by assuming that $p_u = 0$. We can write
$$
(p_y + 1-p_{y+1}) h_{u,y} = p_y h_{u,y-1} + (1-p_{y+1})h_{u,y+1}, \hbox{ for } 0<y<u.
$$
Now, this can be equivalently written as
$$
(1-p_{y+1})(h_{u,y+1}-h_{u,y}) = p_y(h_{u,y}-h_{u,y-1}), \hbox{ for } 0 < y < u.
$$
Let $\delta_{u,y} := h_{u,y+1} - h_{u,y}$. Then, note
$$
\delta_{u,y} = \frac{p_y \cdots p_1}{(1-p_{y+1})\cdots (1-p_2)}\delta_{u,0}
$$
where $\delta_{u,0} = h_{u,1}$. Since $h_{u,y} = \delta_{u,y-1} + \cdots + \delta_{u,0}$ and $h_{u,u} = 1$, we have
$$
\left(1 + \sum_{z=1}^{u-1}\frac{p_z\cdots p_1}{(1-p_{z+1})\cdots (1-p_2)}\right)\delta_{u,0} = 1.
$$
Hence, we have
\begin{equation}
h_{u,k} = \frac{1 + \sum_{z=1}^{k-1}\frac{p_z\cdots p_1}{(1-p_{z+1})\cdots (1-p_2)}}{1 + \sum_{z=1}^{u-1}\frac{p_z\cdots p_1}{(1-p_{z+1})\cdots (1-p_2)}}, \hbox{ for } 1\leq k < u.
\label{equ:hki}
\end{equation}

For the special when all the transition probabilities of $Y$ of values in $(0,1)$ are equal to $1/2$, we have
\begin{equation}
h_{u,k} = \frac{2k}{2u-1}, \hbox{ for } 1\leq k < u.
\label{equ:hki1}
\end{equation}

\paragraph{Case $k > u$} In this case, vertex $u$ has at least one informative interaction if, and only, if there exits a time step $t$ such that $I_t= u$ and $Y_t = u$. Let $\tilde{Y}$ be a Markov chain with state space $\{u-1,u,\ldots, n\}$ and transition probabilities for vertices  $u < v \leq n$ corresponding to those of $Y$ and by definition
\begin{eqnarray*}
\P[\tilde{Y}_{t+1} = u+1 \mid \tilde{Y}_t = u] &=& \frac{1}{n}(1-p_{u+1})\\
\P[\tilde{Y}_{t+1} = u-1 \mid \tilde{Y}_{t} = u] & = & \frac{1}{n}\\
\P[\tilde{Y}_{t+1} = u \mid \tilde{Y}_{t} = u] & = & 1-\frac{1}{n}(1-p_{u+1})-\frac{1}{n}
\end{eqnarray*}
and
$$
\P[\tilde{Y}_{t+1}=u-1 \mid \tilde{Y}_t = u-1] = 1.
$$

The value of $h_{u,k}$ corresponds to $\tilde{Y}$ hitting state $u-1$ by starting from state $k$. We have boundary conditions $h_{u,u-1} = 1$ and $h_{u,n} = 0$. By same arguments as before, for $u-1 < y < n$,
$$
h_{u,y} = \frac{1}{n}p_y h_{u,y-1} + \frac{1}{n}(1-p_{y+1})h_{u,y+1} + \left(1-\frac{1}{n}p_y - \frac{1}{n}(1-p_{y+1})\right)h_{u,y}
$$
where we abuse the notation by assuming $p_u = 1$.

Again, it follows
$$
(1-p_{y+1})(h_{u,y+1}-h_{u,y}) = p_y (h_{u,y} - h_{u,y-1}), \hbox{ for } u - 1 < y < n
$$
and
$$
\delta_{u,y} = \frac{p_y \cdots p_u}{(1-p_{y+1})\cdots (1-p_{u+1})}\delta_{u,u-1}, \hbox{ for } u-1 < y < n.
$$

Since $h_{u,y} = -(\delta_{u,y} + \delta_{u,y+1} + \cdots + \delta_{u,n-1})$ and $h_{u-1,u} = 1$, we obtain
\begin{equation}
h_{u,k} = \frac{\sum_{z=k}^{n-1}\frac{p_z\cdots p_u}{(1-p_{z+1})\cdots (1-p_{u+1})}}{1+\sum_{z=i}^{n-1}\frac{p_z\cdots p_u}{(1-p_{z+1})\cdots (1-p_{u+1})}}, \hbox{ for } u < k < n.
\label{equ:hik}
\end{equation}

In particular, when all the transition probabilities of $Y$ of values in $(0,1)$ are equal to $1/2$, we have
\begin{equation}
h_{i,k} = \frac{2(n-k)}{2(n-i)+1}, \hbox{ for } i < k < n.
\label{equ:hik1}
\end{equation}

\paragraph{Case $k=i$} In this case, we have
$$
h_{k,k} = \frac{1}{n} + \frac{1}{n}(1-p_{k+1})h_{k,k+1} + \left(1-\frac{1}{n} -\frac{1}{n}(1-p_{k+1})\right)h_{k,k}
$$
Hence,
\begin{equation}
h_{k,k} = \frac{1 + (1-p_{k+1})h_{k,k+1}}{2-p_{k+1}}.\label{equ:hkk}
\end{equation}

In particular, when all the transition probabilities of $Y$ of values in $(0,1)$ are equal to $1/2$, we have
\begin{equation}
h_{k,k} = \frac{2}{3}\left(1 + \frac{(n-k)}{2(n-k)+1}\right).\label{equ:hk1}
\end{equation}

We next discuss the results of the above analysis for the special case when the transition probabilities of $Y$ of value in $(0,1)$ are equal to $1/2$. From (\ref{equ:hki}), we observe that vertex $n$ has at least one informative interaction with probability 
$$
h_{n,k} = \frac{2n}{2n-1}\frac{k}{n}.
$$
For large $n$, we have $h_{n,k} \sim k/n$. It follows that vertex $n$ has a diminishing probability of having at least one informative interaction provided that $k/n = o(1)$. For instance, if $k$ is a constant, then $h_{n,k} = \Theta(1/n)$.

Let $N$ be the number of vertices  with at least one informative interaction. We have
$$
\E[N] = \sum_{u=1}^n h_{u,k}.
$$

Note the following elementary identity
$$
S_n : = \sum_{u=1}^n \frac{1}{2u-1} = \frac{1}{2}(2H_{2n}-H_n)
$$
and note that
$$
T_n :=\sum_{i=1}^n \frac{1}{2n+1} = S_{n+1}-1.
$$

We first compute
\begin{eqnarray*}
\sum_{u=k+1}^n h_{u,k} &=& 2k \sum_{u=k+1}^n \frac{1}{2u-1} \\
&=& 2k(S_{n}-S_k)\\
&=& k(2H_{2n}-H_n - 2H_{2k} + H_k)\\
&=& k\log(n/k) (1+o(1)).
\end{eqnarray*}

Then, we compute
\begin{eqnarray*}
\sum_{u=1}^{k-1} h_{u,k} &=& 2(n-k) \sum_{u=1}^{k-1}\frac{1}{2(n-u)+1}\\
&=& 2(n-k)\sum_{v=n-k+1}^{n-1} \frac{1}{2v+1}\\
&=& 2(n-k) (T_{n-1}-T_{n-k})\\
&=& 2(n-k) (S_n - S_{n-k+1})\\
&=& (n-k)(2H_{2n} - H_n - 2H_{2(n-k+1)} + H_{n-k+1})\\
&=& (n-k)\log(n/(n-k+1))(1+o(1)).
\end{eqnarray*}

It follows that
$$
\E[N] = k\log\left(\frac{n}{k}\right) + (n-k)\log\left(\frac{n}{n-k+1}\right) + \Theta(1).
$$


\bibliographystyle{abbrvnat}
\bibliography{references}

\end{document}